\documentclass[12pt,a4paper]{article}

\usepackage{graphicx}
\usepackage{float}
\usepackage{enumerate}
\usepackage{soul}
\usepackage{xcolor}

\usepackage{amsmath,amssymb,eulervm,doi}

\usepackage{amsthm}
\usepackage{amsmath}
\newtheorem{theorem}{Theorem}
\newtheorem{lemma}[theorem]{Lemma}

\newtheorem{remark}{Remark}
\newtheorem{property}{Property}
\newtheorem{definition}{Definition}

\title{Micro-macro Parareal, from ordinary differential equations to stochastic differential equations and back again}

\author{Ignace Bossuyt 
\thanks{Department of Computer Science,
Celestijnenlaan 200 A, box 2402,
3001 LEUVEN, Belgium.  
\protect\url{mailto: ignace.bossuyt1@kuleuven.be},
\protect\url{https://orcid.org/0000-0003-2066-400X}}
\and Stefan Vandewalle
\thanks{Department of Computer Science,
Celestijnenlaan 200 A, box 2402,
3001 LEUVEN, Belgium.  \protect\url{mailto: stefan.vandewalle@kuleuven.be} ,
\protect\url{https://orcid.org/0000-0002-8988-2374}}
\and Giovanni Samaey
\thanks{Department of Computer Science,
Celestijnenlaan 200 A, box 2402,
3001 LEUVEN, Belgium.  \protect\url{mailto: giovanni.samaey@kuleuven.be} ,
\protect\url{https://orcid.org/0000-0001-8433-4523}}
}

\usepackage[square,sort,comma,numbers]{natbib}
\newcommand{\norm}[1]{\left\|#1\right\|}

\NewDocumentCommand{\mycite}{o m}{\citeauthor{#2} \IfValueTF{#1}{\cite[#1]{#2}}{\cite{#2}}}

\begin{document}

\maketitle

\begin{abstract}
We are concerned with the micro-macro Parareal algorithm for the simulation of initial-value problems. 
In this algorithm, a coarse (fast) solver is applied sequentially over the time domain, and a fine (time-consuming) solver is applied as a corrector in parallel over smaller chunks of the time interval. 
Moreover, the coarse solver acts on a reduced state variable, which is coupled to the fine state variable through appropriate coupling operators.
We first provide a contribution to the convergence analysis of the micro-macro Parareal method for multiscale linear ordinary differential equations (ODEs). 
Then, we extend a variant of the micro-macro Parareal algorithm for scalar stochastic differential equations (SDEs) to higher-dimensional SDEs.


\end{abstract}


\section{Introduction}
Parallel-in-time methods aim to pararallelise the time-simulation of initial-value problems over the time domain.
In the Parareal algorithm, first proposed by \mycite{lions_resolution_2001_AMS}, a fine time-propagator (accurate but with long wall-clock time) is applied on small chuncks of the time domain in parallel. A correction is provided by a coarse time-propagator (fast but with reduced accuracy), that is applied sequentially over the complete time domain.

More specifically, let $u(t): \mathbb R \rightarrow \mathbb R^d$ be a function of time $t$, let $f :\mathbb R^d \rightarrow \mathbb R^d$ be a  function, and $\alpha \in \mathbb R^d$ be a constant vector, then the Paraeal method applies to ODE problems of the form
\begin{equation}
du/dt = f(u) \qquad u(0) = \alpha.
\label{IVP_formulation}
\end{equation} 
on a time interval $[0, \, T]$.
The time domain is divided in $N$ subintervals of equal length $\Delta t = T/N$, and the elements of the time grid are denoted with $t_n = n \Delta t$, for $0 \leq n \leq N$. 
Let $\mathcal F$ be a fine solver $\mathcal F: \mathbb R^d \rightarrow \mathbb R^d$, numerically evolving an approximate solution to equation \eqref{IVP_formulation} over a time chunk of time $\Delta t$, and similarly let $\mathcal C: \mathbb R^d \rightarrow \mathbb R^d$ be a coarse solver.

The Parareal iterate $u^k_n$ with index $k$ for iteration number $k \geq 0$ and $n$ for time index $0 \leq n \leq N-1$, obeys this double recursion:
\begin{equation}
\begin{cases} 
	u^{k}_{n+1} = \mathcal C (u^{k}_{n}),
	& \qquad k=0, \\
    u^{k+1}_{n+1} = \mathcal F(u^{k}_{n}) + \mathcal C( u^{k+1}_{n}) - \mathcal C(u^{k}_{n}), 
    & \qquad k > 0.
\end{cases}
    \label{classical_Parareal}
\end{equation}
Typically, a reference solution $u_n$ is defined as the result of the sequential application of the fine solver: $u_0 = \alpha$ and $u_{n+1} = \mathcal F(u_n)$ for $0 \leq n \leq N-1$. 
The approximations by the Parareal algorithm then possesses a finite-termination property, namely $u^k_n = u_n$ for $k \geq n$ (see Theorem 3.1 and Remark 4.7 in \mycite{gander_analysis_2007}).
In practice, Parareal is also applicable for nonautonomous problems and nonuniform time grids.

The convergence of Parareal for linear scalar differential equations is given by the following theorem by \mycite[corrolary 4.6 and theorem 4.9]{gander_analysis_2007}.
\begin{lemma}
\label{theorems_from_Gander_Vandewalle}
Let $\mathcal F(u) = fu$ be a scalar fine propagator and $\mathcal C(u) = gu$ be a scalar coarse propagator with $f,g \in \mathbb R$ and $|g|< 1$. 
Then, the maximum error $E^{k}_{\mathrm{max}} = \max_{1\leq n \leq N} |u^k_n - u_n|$, 
satisfies a superlinear bound
\begin{equation}
E^{k}_{\mathrm{max}}
\leq
\left| f-g \right|^k 
\binom{N-1}{k}
E^{0}_{\text{max}},
\label{superlinear_bound_old}
\end{equation} 
and a linear bound: 
\begin{equation}
E^{k}_{\mathrm{max}}
\leq 
\left( \frac{\left| f-g \right|}{1-\left| g \right|} \right)^k 
E^{0}_{\text{max}}.
\label{linear_bound_old}
\end{equation}
\end{lemma}

In micro-macro Parareal, which is a generalisation of Parareal, the coarse propagator does not act on the original state variable $u \in \mathbb R^d$ (micro state), but instead on a reduced version $U \in \mathbb R^r$ (macro state). 
These micro and macro states are coupled.
A restriction operator $\mathcal R: \mathbb R^d \rightarrow \mathbb R^r$ extracts macro information from a micro state.
A lifting operator $\mathcal L$ provides a unique micro state $u$, consistent with a given macro state $U$. That is, $\mathcal{L}: \mathbb R^r \rightarrow \mathbb R^d: u = \mathcal L(U)$ such that $\mathcal R(u) = U$.

\mycite{blouza_parallel_2010} proposed one possible micro-macro Parareal algorithm (see also \mycite[Equation 3.21]{Legoll2013}). 
The micro initial condition for all iterations $k$ is $u^k_0 = u_0$ and the macro initial condition equals $U^k_0 = \mathcal R (u_0)$.
In the zeroth iteration, one computes $U^0_{n}$ and $u^0_{n}$ sequentially for $n=1 \hdots N$:
\begin{equation}
\begin{aligned} 
U_{n+1}^{0}  &= \mathcal{C} (U_n^{0}), \\ 
u_{n+1}^{0} &= \mathcal{L}(U_{n+1}^{0}),
\end{aligned}
\label{mM_parareal_maday_0}
\end{equation}
In the subsequent iterations:
\begin{equation}
\begin{aligned}
u_{n+1}^{k+1} 
&=  
\mathcal{F} (u_n^k)
+
\mathcal L \left [ \mathcal{C} (\mathcal{R} (u_n^{k+1}) ) 
- 
\mathcal{C}( \mathcal{R} ( u_n^k) ) 
\right] \\
U^{k+1}_{n+1} 
&= 
\mathcal R (u^{k+1}_{n+1})
\end{aligned}
\label{mM_parareal_maday_other_k}
\end{equation}
Depending on the context, however, a lifting operator may involve costly computations.  
In that case, a matching operator can be used instead. 
A matching operator $\mathcal M$ outputs a micro state $u$ consistent with a given macro state $U$, based on prior information about the micro state $\hat u$, that is, $\mathcal M: (\mathbb R^r ,\, \mathbb R^d) \rightarrow \mathbb R^d: \mathcal M(U, \hat u) = u$ such that $\mathcal R(u) = U$. 

Another micro-macro Parareal algorithm, using a matching operator, was proposed by \mycite[Algorithm 2]{Legoll2013}. 
As in the previous algorithm, the micro initial condition for all iterations $k$ is $u^k_0 = u_0$ and the macro initial condition equals $U^k_0 = \mathcal R (u_0)$.
It starts with \eqref{mM_parareal_maday_0}, and replaces \eqref{mM_parareal_maday_other_k} by
\begin{equation}
\begin{aligned}
U_{n+1}^{k+1} 
&= \mathcal{C} (U_n^{k+1}) 
+ \mathcal{R} ( \mathcal{F} (u_n^k))
- \mathcal{C}( U_n^k) \\     
u_{n+1}^{k+1} &= \mathcal{M}(U_{n+1}^{k+1}, \mathcal{F} (u_n^k)).
\end{aligned}
\label{mM_parareal_other_k}
\end{equation}
When we refer to micro-macro Parareal in the rest of this text, we refer to this variant.

\begin{remark}[Micro and macro error]
In the study of convergence properties of the micro-macro Parareal algorithm, one may be interested in the error on the micro variable $u_n^k$ or in the error on the macro variable $U_{n}^k$. 
We denote these errors $e^{u,k} = u_{n}^k - u_n$ and $e^{U,k} := U_{n}^k - U_n = U^k_n - \mathcal{R}(u^k_n)$, respectively.
In practice, this choice is influenced by the context. 
See also remark \ref{remark_impossible_everything_in_terms_of_macro_error}.
\end{remark}

\begin{property}
\label{properties_micro_macro_Parareal}
If $\mathcal R$ and $\mathcal M$ are chosen such that, for any $U$ and $u$, $\mathcal R (\mathcal M(U, u)) = U$, then the following properties of micro-macro Parareal \eqref{mM_parareal_maday_0}-\eqref{mM_parareal_other_k} hold:

\begin{enumerate}[(i)]
\item A micro-macro consistency property, that is $U^k_n = \mathcal R (u^k_n)$, $\forall n,k \geq 0$ (see \mycite[last paragraph of section 3.2]{Legoll2013}).

\item If, in addition, $\mathcal M ( \mathcal R (u), u) = u$, then, a finite termination property holds: $u^k_n = u_n$ for all $k \geq n$ (see \mycite[Theorem 12]{Legoll2013}). From (i), it also holds that $U^k_n = U_n = \mathcal R (u_n)$. 
\end{enumerate} 
\end{property}

\noindent
Under the same assumptions of property \ref{properties_micro_macro_Parareal}, it holds that
\begin{enumerate}[(i)]
\item $e^{U,k}_n = \mathcal R (e^{u,k}_n)$, $\forall n,k \geq 0$,
\item $e^{u,k}_n = 0$ and $e^{U,k}_n = 0$  for all $k \geq n$.
\end{enumerate}

\paragraph{Overview of paper.}
The rest of this paper is organised as follows.
We first present a convergence bound for the micro-macro Parareal algorithm applied to a special multiscale ODE in section \ref{section_micro_macro_ODE}. 
We present a new upper bound on the Parareal error, explicitly with respect to the iteration number. This bound is valid for a specific class of (multiscale) ODEs.

Then, in section \ref{section_micro_macro_SDE}, we extend the MC-moments Parareal algorithm by \mycite{bossuyt_monte_carlomoments_2023}, designed for the simulation of scalar (McKean-Vlasov) stochastic differential  equations (SDEs) to higher dimensional SDEs.

In both parts, we provide some numerical experiments to illustrate our theory and methods. 
The Julia implementation for all the numerical experiments is available in a public repository from \cite{ignace_software_paper_ANZIAM}.

\section{Convergence theory of Micro-macro Parareal for linear multiscale ODEs}
\label{section_micro_macro_ODE}
In this section, we consider a two-dimensional 
linear multiscale ODE

\paragraph{Test system and assumptions.}
The ODE that we consider models the evolution of a state vector $u = (x,y)$ where $x \in \mathbb R$ is a slow variable and $y \in \mathbb R$ is a fast variable. 
It contains the parameters $\alpha$, $\beta, \delta \in \mathbb R$. 
\begin{equation}
\begin{aligned}
\frac{d}{dt} 
\left[ 
\begin{matrix}
x \\ y
\end{matrix}
\right] 
= 
\left[ 
\begin{matrix}
\alpha & \beta \\
0 & \delta	
\end{matrix}
\right]
\left[ 
\begin{matrix}
x \\ y
\end{matrix}
\right] 
\qquad 
\qquad
&
\left[ 
\begin{matrix}
x(0) \\ y(0)
\end{matrix}
\right],
= 
\left[ 
\begin{matrix}
x_0 \\ y_0
\end{matrix}
\right].
\end{aligned}
\label{multiscale_ODE}
\end{equation}
Here, we assume that $\delta$ is large compared to $\alpha$: $|\delta| \gg \alpha$ such that the fast variable reaches its equilibrium much faster in time than the slow variable.
We assume that $\delta < 0$ and also that $\delta \neq \alpha$. This ensures that the fast variable decays to zero.
If one increases the absolute value of $\delta$, the numerical simulation of this system gets more  expensive due to the stiffness in the $y$ variable. 

The (exact) solution to equation \eqref{multiscale_ODE} equals $u(t+\Delta t) = \mathcal A(\Delta t) u(t)$
where 
\begin{equation}
\mathcal A (\Delta t) 
=
\left[
\begin{matrix}
e^{\alpha \Delta t} & (e^{\delta \Delta t} - e^{\alpha \Delta t}) \frac{\beta}{\delta - \alpha} \\
0 & e^{\delta \Delta t}
\end{matrix}
\right].
\label{definition_mathcal_A}
\end{equation}
This is a special case of the linear ODE considered in \mycite{Legoll2013}, namely a two-dimensional case where the leftmost entry in the matrix $\mathcal A$ equals zero.

\paragraph{Reduced model.}
Now we build a one-dimensional reduced model, describing only the dynamics of a slow variable $U$ (as an approximation to $x$). 
To this aim, we introduce a modified decay rate $\bar{\alpha} \approx \alpha$. 
The evolution of the reduced variable $U$ is given as follows:
\begin{equation}
\frac{dU}{dt} = \bar \alpha U ,
\qquad \qquad
U(0) = \bar{x}_0.
\label{reduced_ODE}
\end{equation}
where $\bar{x}_0$ is the initial condition of the reduced model. We denote the effect of a numerical time-stepper on equation \eqref{reduced_ODE} with a scalar $G \in \mathbb R$ as $G \approx e^{\alpha \Delta t}$.



\begin{definition}[Operators for micro-macro Parareal on linear multiscale ODEs]
\label{definition_micro_macro_Parareal_ODEs}
Following \mycite{Legoll2013}, we choose the micro variable $u = ( x, \, y )$ (corresponding to the multiscale ODE \eqref{multiscale_ODE}) and the macro variable $U$ (corresponding to the reduced model \eqref{reduced_ODE}). Further, we choose

\begin{itemize}
\item a fine propagator $\mathcal F: 
\mathbb R^2 \rightarrow \mathbb R^2: 
\mathcal F(u) := \mathcal A u$ (see the stiff model \eqref{multiscale_ODE}) and 

\item a coarse propagator $\mathcal C: 
\mathbb R \rightarrow \mathbb R: 
\mathcal C(U) := GU$ (see the reduced model \eqref{reduced_ODE}),

\item a restriction operator
$ \mathcal R: 
\mathbb R^2 \rightarrow \mathbb R: 
\mathcal R ((x, \, y))
:= x$, 

\item a matching operator  
$ \mathcal M: 
(\mathbb R, \mathbb R^2) \rightarrow \mathbb R^2:
\mathcal M(U, (x, \, y)) :=
(U, \, y)$

\item a lifting operator 
$ \mathcal L: 
\mathbb R \rightarrow \mathbb R^2:
\mathcal L(U) := 
(U, \, 0) $.
\end{itemize}
We divide the time-interval into $N$ subintervals of equal length.
\end{definition}
The coarse and fine solver are linear, as the ODE problem and its reduced model are linear.
The above choices suffer from the issue of initial slip, for more information see for instance \mycite[Appendix A]{van_kampen_elimination_1985}. 
We postpone a discussion of this issue to section \ref{convergence_bound_numerical_experiment}, where we also give some numerical experiments.

\paragraph{Relation to existing work.}
\mycite{Legoll2013} study the convergence of micro-macro Parareal with respect to the time-scale separation between the fast and slow variable for a more general multiscale linear ODE than the ODE \eqref{multiscale_ODE}, i.e., with a bottom left term of the matrix $\mathcal A$ possibly different from zero. 
Their bound contains a constant which depends on the iteration number. That makes it difficult to predict the error of micro-macro Parareal \textit{a priori} as a function of the iteration number. 
In this work, we wish to obtain a convergence bound directly with respect to the iteration number. 

\subsection{A non-tight error bound}
In this section, we first interpret micro-macro Parareal for the multiscale ODE as a special case of classical Parareal.
We then explain why it is not possible to readily apply the convergence theory from \mycite{gander_analysis_2007}, because the coarse and fine propagators do no commute.
Then we derive a bound on the error of micro-macro Parareal based on the generating functions methodology developed by \mycite{gander_unified_2023}. 

\paragraph{Preparatory considerations}

\begin{definition}[Introduction of symbols]
In order to make notation shorter, we introduce $F = e^{\alpha \Delta t}$
and $b = (e^{\delta \Delta t} - e^{\alpha \Delta t}) \frac{\beta}{\delta - \alpha}$.
We also define the matrix $\mathcal B$ as follows:
\begin{equation}
\mathcal B 
= \left[ \begin{matrix}
G & 0 \\
0 & 0
\end{matrix} \right].
\label{definition_mathcal_B}
\end{equation}
\end{definition}

We first present a lemma about the interpretation of the micro-macro Parareal iterates as iterates of classical Parareal, with proof in appendix \ref{proof_micro_macro_error}.
\begin{lemma}[Interpretation of micro-macro Parareal as classical Parareal]
\label{solution_micro_macro_Parareal}
The use of micro-macro Parareal \eqref{mM_parareal_maday_0}-\eqref{mM_parareal_other_k} with the operators from definition \ref{definition_micro_macro_Parareal_ODEs}
is equivalent to the application of classical Parareal with a fine propagator $\mathcal F: \mathbb R^2 \rightarrow \mathbb R^2: \mathcal F(u) = \mathcal Au$ and a coarse  propagator $\mathcal C: \mathbb R^2 \rightarrow \mathbb R^2: \mathcal C(u) = \mathcal Bu$:
\begin{equation}
\begin{aligned}
u^0_{n+1} 
&= 
\mathcal B u^0_n 
\\
u^{k+1}_{n+1} 
&=
(\mathcal A - \mathcal B)
u^{k}_{n}
+ 
\mathcal{B} u^{k+1}_n.
\end{aligned}
\label{micro_macro_Parareal_as_classical}
\end{equation}
\end{lemma}

\begin{lemma}[Error equation for micro-macro Parareal iterates]
\label{error_micro_macro_Parareal}
Consider micro-macro Parareal, defined in
equations \eqref{mM_parareal_maday_0}-\eqref{mM_parareal_other_k}, with coarse and fine time-stepping operators and coupling operators from definition \ref{definition_micro_macro_Parareal_ODEs}.
The error on the micro state $u^k_n$ in iteration $k$ and at time point $t_n$, namely $e^{u,k}_n = u^k_n - u_n$, satisfies for $0 \leq n \leq N-1$
\begin{equation}
\begin{aligned}
e^{u,0}_{n+1} 
&= B u^0_n - A u_n = B^nu_0 - A^n u_0
\\
e^{u,k+1}_{n+1} 
&=
(\mathcal A - \mathcal B)
e^{u,k}_{n}
+ 
\mathcal{B} 
e^{u,k+1}_n.
\end{aligned}
\label{error_equation_micro_macro_Parareal}
\end{equation}
\end{lemma}
\begin{proof}
Substracting the reference solution $u_n = \mathcal A^n u_0$ on both sides of equation \eqref{micro_macro_Parareal_as_classical} leads to \eqref{error_equation_micro_macro_Parareal}.
\end{proof}

We denote the error on the slow variable with $e^{x,k}_n = x^k_n - x_n$ and the error on the fast variable with $e^{y,k}_n = y^k_n - y_n$.

\begin{remark}
\label{remark_impossible_everything_in_terms_of_macro_error}
It is, in general, not possible to derive a closed form solution to equation for the evolution of the macro error. 
Indeed, using the linearity of $\mathcal R$
\begin{equation}
\begin{aligned}
e^{U,k+1}_{n+1} 
= 
\mathcal R \left(
e^{u,k+1}_{n+1} \right) 
&=
\mathcal R \left(
(\mathcal A - \mathcal B)
e^{u,k}_{n}
+ 
\mathcal{B} 
e^{u,k+1}_n
\right) 
\\
&=
\mathcal R \left(
(\mathcal A - \mathcal B)
e^{u,k}_{n}
\right)
+ 
\mathcal R \left(
\mathcal{B} 
e^{u,k+1}_n
\right)
\\
&=
\mathcal R \left(
(\mathcal A - \mathcal B)
e^{u,k}_{n}
\right)
+ 
G
e^{U,k+1}_n.
\end{aligned}
\end{equation}
For our model problem, $\mathcal R \left(
(\mathcal A - \mathcal B)
e^{u,k}_{n}
\right) \neq (\mathcal A - \mathcal B) \mathcal R \left(
e^{u,k}_{n}
\right)$, thus it is not possible to write an expression for the evolution of the macro error in terms of the macro error alone.
\end{remark}

\paragraph{Summary of derivation of convergence bounds from \mycite{gander_analysis_2007}.}
A first approach to obtain an upper bound on the error $e^{u,k}_n$ could be to generalise the analysis for a linear scalar test ODE used by \mycite{gander_analysis_2007} to multidimensional linear ODEs. 
In this approach, we define the vector 
$
\mathbf e^{u,k} := \left[
\begin{matrix} 
e^{u,k}_1 & \hdots & e^{u,k}_N 
\end{matrix} \right]$.

Let us now define the matrices $M \in \mathbb R^{2N \times 2N}$ and $C \in \mathbb R^{2N \times 2N}$ as 
\begin{equation}
M = \left[ 
\begin{matrix} 
I & 0 & \hdots & 0 & 0 \\ 
-\mathcal B & I & \hdots & 0 & 0 \\
\hdots \\
0 & -\mathcal B & \hdots & I & 0 \\
0 & 0 & \hdots & -\mathcal B & I \\
\end{matrix}
\right];
\quad 
C = \left[ 
\begin{matrix}
0 & 0 & \hdots & 0 & 0 \\ 
\mathcal{A-B} & 0 & \hdots & 0 & 0 \\
0 & \mathcal{A-B} & \hdots & 0 & 0 \\
\hdots & & & & \hdots \\
0 & 0 & \hdots & \mathcal{A-B} & 0 \\
\end{matrix}
\right]
\label{definition_matrices_M_and_C}
\end{equation}
Then, $\mathbf e^{u,k+1}$ is related to $\mathbf e^{u,k}$ through this equation: 
\begin{equation}
\begin{aligned} 
M \mathbf e^{u,k+1} &= C \mathbf e^{u,k}, \mathrm{thus}
\\
\mathbf e^{u,k+1} &= M^{-1}C \mathbf e^{u,k}.
\end{aligned}
\end{equation}

\begin{lemma}
\label{lemma_commuting_matrices}
If $\mathcal B$ and $\mathcal A- \mathcal B$ commute, then 
the matrices $M^{-1}$ and $C$ commute.
Furthermore, let us define the matrices $H \in \mathbb R^{2N \times 2N}$ and $D \in \mathbb R^{2N \times 2N}$:
\begin{equation}
H = \left[ 
\begin{matrix} 
0 & 0 & \hdots & 0 & 0 \\ 
I & 0 & \hdots & 0 & 0 \\
\mathcal B & I & \hdots & 0 & 0 \\
\hdots \\
\mathcal B^{N-2} & \mathcal B^{N-1} & \hdots & I & 0 \\
\end{matrix}
\right];
\quad 
D = \left[ 
\begin{matrix}
\mathcal{A-B} & 0 & \hdots & 0 & 0 \\ 
0 & \mathcal{A-B} & \hdots & 0 & 0 \\
0 & & \hdots & 0 & 0 \\
\hdots & & & & \hdots \\
0 & 0 & \hdots & 0 & \mathcal{A-B} \\
\end{matrix}
\right].
\label{definition_matrices_N_and_D}
\end{equation}
Then, $M^{-1}C = HD$.
If the matrices $M^{-1}$ and $C$ commute, then the matrices $H$ and $D$ also commute.
\end{lemma}
\begin{proof}
See appendix \ref{appendix_proof_commuting_matrices}.
\end{proof}

\noindent
If the matrices $M^{-1}$ and $C$ commute, then it is possible to write
\begin{equation}
e^k = (M^{-1}C)^k e^0 = (HD)^k e^0 = H^k D^k e^0,
\end{equation}
and thus
\begin{equation}
\norm{e^k} 
\leq 
\norm{H^k} \norm{D^k} \norm{e^0}.
\end{equation}
The superlinear upper bound from \mycite{gander_analysis_2007}, given in lemma \ref{theorems_from_Gander_Vandewalle}, is obtained through bounding $\norm{H^k}$ and $\norm{C^k}$.

If the matrices $M^{-1}$ and $C$ do not commute, it holds that
\begin{equation}
\norm{e^k} 
\leq
\norm{H}^k \norm{D}^k.
\end{equation}
The linear bound in equation \eqref{linear_bound}, as obtained by \mycite{gander_analysis_2007}, follows from bounding $\norm{H}$ and $\norm{D}$.
In our test problem, the matrices $H$ and $C$ indeed do not commute, since the matrices $\mathcal A$ and $\mathcal B$ that appear as blocks in $M$ and $C$, do not commute. 

In conclusion, the superlinear bound requires commutativity of $H$ and $D$, while a linear bound can be derived without this assumption. It is our aim, however, to derive an effective superlinear bound for our multiscale test problem.

%

\paragraph{Using generating functions.}
An alternative idea is to use the generating function methodology used by \mycite{gander_unified_2023}. One then proceeds as follows. From equation \eqref{error_equation_micro_macro_Parareal} it can be written that 
\begin{equation}
\| e^{u,k+1}_{n+1} \| 
\leq
\| \mathcal A - \mathcal B \|
\| e^{u,k}_{n} \|
+ 
\| \mathcal{B} \| 
\| e^{u,k+1}_n \|
\label{equation_nontight_bound}
\end{equation}
where $\| \cdot \|$ denotes a vector norm or the induced matrix norm. A proper matrix norm satisfies the submultiplicativity property $\| A B \| \leq  \| A \| \| B \|$ for arbitrary matrices $A$, $B$. 
Starting from equation \eqref{equation_nontight_bound}, the generating function method delivers this bound on $\norm{e^{u,k}_n}$, taken from \mycite[equation 3.12]{gander_unified_2023}:
\begin{equation}
\| e^{u,k}_{n} \|
\leq 
\frac{\|\mathcal A- \mathcal B\|}{(k-1)!} \sum_{i=0}^{n-k} \prod_{l=1}^{k-1}(i+l) \norm{\mathcal B}^i \max_{1\leq n \leq N} \norm{e^{u,0}_n}.
\label{equation_nontight_bound_result}
\end{equation}
A bound for $e^{x,k}_{\mathrm{max}} = \max_{1 \leq n \leq N} |x^k_n - x_n|$ can be constructed by writing 
$e^{x,k}_{\mathrm{max}} = \max_{1 \leq n \leq N} |e^{x,k}_n| \leq \norm{e^{u,k}_{\mathrm{max}}}$.
This bound is not tight, however, as is illustrated in Section \ref{convergence_bound_numerical_experiment}.

\subsection{Exploiting the structure of the multiscale ODE to obtain tighter covergence bounds}
In this section, we derive an upper bound for $e_{\mathrm{max}}^{x,k}$, which exploits the special structure of the linear multiscale ODE \eqref{multiscale_ODE}.
In this analysis, we first reformulate micro-macro Parareal for the multiscale ODE \eqref{multiscale_ODE} as classical Parareal for a scalar nonhomogeneous ODE.
This allows us to extend the convergence analysis originally presented by \mycite{gander_analysis_2007}.
First we provide a lemma on the convergence of the (autonomous) fast variable.

\begin{lemma}[Error of the fast variable]
\label{lemma_error_fast_variable}
Let $e^{y,k}_{\mathrm{max}} = \max_{1 \leq n \leq N} |y^k_n - y_n|$ be the maximum error (over all $0 \leq n \leq N$) in the micro-macro Parareal approximation (equations \eqref{mM_parareal_maday_0}-\eqref{mM_parareal_other_k}) of the fast variable, with the operators defined in definition \ref{definition_micro_macro_Parareal_ODEs}.
Then,
\begin{equation}
e^{y,k}_{\mathrm{max}} 
\leq 
e^{k \delta \Delta t} 
e^{y,0}_{\mathrm{max}}.
\label{error_fast_variable}
\end{equation}
\end{lemma}

\begin{proof}
From equation \eqref{error_equation_micro_macro_Parareal} it holds  that 
\begin{equation}
\begin{aligned}
e^{y,k+1}_{n+1} 
&= y^{k+1}_n  - y_{n+1}\\
&= e^{\delta \Delta t} y^{k}_{n} - e^{\delta \Delta t} y_{n} \\
&= e^{\delta \Delta t} e^{y,k}_{n}.
\end{aligned}
\label{recursion_error_fast_varialbe}
\end{equation}
\paragraph{Algebraic proof.}
The solution of the recursion \eqref{recursion_error_fast_varialbe} equals
\begin{equation}
e^{y,k}_{n} = 
\begin{cases}
e^{k (\delta \Delta t)} e^{y,0}_{n-k}  & 
	\qquad n \geq k \\
0 & \qquad n < k. \\	
\end{cases}
\end{equation}
Then taking the maximum over $n$ leads to equation \eqref{error_fast_variable}.

\paragraph{Equivalent proof in matrix notation.}
Let $L \in \mathbb R^{N \times N}$ be a matrix and let 
$\mathbf e^{y,k} 
= 
\left[ \begin{matrix} e_1^{y,k} & \hdots & e_N^{y,k} \end{matrix} \right]^T$, then we can write
\begin{equation}
\mathbf e^{y,k+1}
= 
\underbrace{
\left[ 
\begin{matrix}
0 & \hdots & \\
e^{\delta \Delta t} & 0 &  & \\
\vdots & & & \\
& & & e^{\delta \Delta t} & 0 & \\
0 & \hdots & & & e^{\delta \Delta t} & 0 
\end{matrix} 
\right]
}_{L}
\mathbf e^{y,k}.
\label{equation_fast_error_matrix_notation}
\end{equation}
The solution of this recursion equals $e^{y,k} = L^{k}e^{y,0}$.
Taking the maximum norm on both sides leads to equation \eqref{error_fast_variable}.
\end{proof}


The linear bound \eqref{error_fast_variable} for the error on the fast variable does not capture the property that after $N$ iterations, it is equal to zero.

\begin{lemma}[Linear and superlinear bound for the error in the slow variable]
\label{lemma_linear_superlinear_bound}
Assume that $\alpha < 0$ and $\delta < 0$ in the multiscale ODE \eqref{multiscale_ODE} and $|G|<1$ in the reduced model \eqref{definition_mathcal_B}.
Let $e^{x,k}_{\mathrm{max}} = \max_{1 \leq n \leq N} |x^k_n - x_n|$ be the error in the micro-macro Parareal approximation (equations \eqref{mM_parareal_maday_0}-\eqref{mM_parareal_other_k}) of the slow variable, with the operators defined in definition \ref{definition_micro_macro_Parareal_ODEs}.

Then, $e^{x,k}_{\mathrm{max}}$ satisfies a linear bound: \begin{equation}
e^{x,k}_{\mathrm{max}}
\leq 
\left( \frac{\left| F-G \right|}{1-\left| G\right|} \right)^k 
e^{x,0}_{\mathrm{max}}
+
\frac{b}{1-\left|G \right|}
\sum_{i=0}^{k-1} 
\left( \frac{|F-G|}{1-|G|} \right)^i 
e^{\delta (k-1-i) \Delta t} e^{y,0}_{\mathrm{max}}.
\label{linear_bound}
\end{equation}
It also satisfies a superlinear bound
\begin{equation}
\begin{aligned}
e^{x,k}_{\mathrm{max}}  
&\leq
\left| F-G \right|^k 
\binom{N-1}{k}
e^{x,0}_{\mathrm{max}} \\
&+
b
\frac{1-|G|^{N-1}}{1-|G|}
\sum_{i=0}^{k-1} 
|F-G|^i 
\binom{N-1}{i}
e^{\delta (k-1-i) \Delta t} e^{y,0}_{\mathrm{max}}.
\end{aligned}
\label{superlinear_bound}
\end{equation}
%
\end{lemma}

\begin{proof}
Equation \eqref{error_equation_micro_macro_Parareal} can be  expanded as
\begin{equation}
\left[ 
\begin{matrix}
e_{n+1}^{x,k+1} \\ e_{n+1}^{y,k+1}
\end{matrix}
\right] 
= 
\left[ 
\begin{matrix}
F & b \\
0 & d	
\end{matrix}
\right]
\left[ 
\begin{matrix}
e_{n}^{x,k} \\ e_{n}^{y, k}
\end{matrix}
\right]
-
\left[
\begin{matrix}
G & 0 \\
0 & 0	
\end{matrix}
\right]
\left[ 
\begin{matrix}
e_{x,n}^{k} \\ e_{n}^{y,k}
\end{matrix}
\right]
+ 
\left[
\begin{matrix}
G & 0 \\
0 & 0	
\end{matrix}
\right]
\left[ 
\begin{matrix}
e_{n}^{x,k+1} \\ e_{n}^{y,k+1}
\end{matrix}
\right],
\label{expanded_error_micro_macro_Parareal}
\end{equation}
where $d = e^{\delta \Delta t}$.
Thus the error on the slow variable, defined as $e^{x,k}_n = x^k_n - x_n$, satisfies
\begin{equation}
e_{n+1}^{x,k+1} = (e^{\alpha \Delta t} - e^{\bar \alpha \Delta t}) e_{n}^{x,k} + e^{\bar \alpha  \Delta t} e_{n}^{x,k+1} + b e_{n}^{y,k}.
\label{slow_error_relation}
\end{equation}
\sloppy
Let us define the matrices $M_x \in \mathbb R^{N \times N}$ and $C_x \in \mathbb R^{N \times N}$
\begin{equation}
M_x = \left[ 
\begin{matrix}
I & \hdots & \\
-G & I &  & \\
\vdots & & & \\
& & & -G &  I & \\
0 & \hdots & & & -G &  I 
\end{matrix} 
\right] ;
\qquad 
C_x = 
\left[ \begin{matrix}
0 & \hdots & & 0\\
(F-G) & 0 & & \vdots \\
\vdots & &  & \\
 &  & 0 & \vdots \\
0 & \hdots & (F-G) & 0 
\end{matrix} 
\right].
\end{equation}
The matrices $M_x$ and $C_x$ in equation \eqref{parareal_error_iteration_matrix} are not the same as the matrices $M$ and $C$ in equation \eqref{definition_matrices_M_and_C}: the latter are block matrices containing the full state space, while the former only describe the slow variable (the evolution of the fast variable is manually added in equation \eqref{parareal_error_iteration_matrix}).
Here, the matrices $M_x$ (or $M_x^{-1}$) and $C_x$ commute.

We now collect all errors at iteration $k$ at different time points $n=1 \hdots N$ in a vector, i.e. 
$\mathbf e^{x,(k)} 
= 
\left[ \begin{matrix} e_1^{x,k} & \hdots & e_N^{x,k} \end{matrix} \right]^T$ 
and 
$\mathbf e^{y,(k)} 
= 
\left[ \begin{matrix} e_1^{y,k} & \hdots & e_N^{y,k} \end{matrix} \right]^T$.
Equation \eqref{slow_error_relation} can then be rewritten as
\begin{equation}
\begin{aligned}
M_x
\mathbf e^{x,(k+1)} 
&= 
C_x
\mathbf e^{x,(k)} 
+ 
b e^{y,(k)}, \, \mathrm{thus} \\
\mathbf e^{x,(k+1)} 
&=  
M_x^{-1} C_x \mathbf \mathbf e^{x,(k)} 
+ b M^{-1}_x \mathbf e^{y,(k)}
\end{aligned}
\label{parareal_error_iteration_matrix}
\end{equation}

We now define the matrices $H_x \in \mathbb R^{N \times N}$ and $D _x\in \mathbb R^{N \times N}$ 
\begin{equation}
H_x = \left[ 
\begin{matrix} 
0 & 0 & \hdots & 0 & 0 \\ 
I & 0 & \hdots & 0 & 0 \\
\hdots \\
G^{N-1} & \hdots & I & 0 & 0 \\
G^{N-2} & G^{N-1} & \hdots & I & 0 \\
\end{matrix}
\right];
\quad 
D_x = \left[ 
\begin{matrix}
F-G & 0 & \hdots & 0 & 0 \\ 
0 & F-G & \hdots & 0 & 0 \\
0 & & \hdots & 0 & 0 \\
\hdots & & & & \hdots \\
0 & 0 & \hdots & 0 & F-G \\
\end{matrix}
\right].
\end{equation}
Then, equation \eqref{parareal_error_iteration_matrix} can be rewritten as 
\begin{equation}
\mathbf e^{x,(k+1)} 
=  
H_x D_x \mathbf e^{x,(k)} + b M^{-1}_x \mathbf e^{y,(k)}.
\end{equation}

From lemma \ref{lemma_error_fast_variable} (equation \eqref{equation_fast_error_matrix_notation}), it holds that $e^{y,(k)} = L^k e^{y,(0)}$.
The solution to the recursion in equation \eqref{parareal_error_iteration_matrix} equals, according to lemma \ref{lemma_matrix_recursion} (see appendix \ref{appendix_matrix_recursion}),
\begin{equation}
\begin{aligned}
\mathbf e^{x,(k)} 
= 
&[H_xD_x]^k 
\mathbf e^{(0)} + 
b
\sum_{i=0}^{k-1} 
[H_xD_x]^{i} 
M^{-1}_x
L^{k-1-i} 
\mathbf e^{y,(0)}.
\end{aligned}
\label{solution_matrix_recursion_in_proof}
\end{equation}
This equation is an exact formula for $e^{x,(k)}$. Next, we derive an upper bound for $\max_{1 \leq n \leq N} |e^{x,k}_n| = \norm{e^{x,(k)}}_{\infty}$.

Since $F$ and $G$ are scalars, $M^{-1}_x$ and $C_x$ commute (and, from lemma \ref{lemma_commuting_matrices}, so do $H_x$ and $D_x$), we can proceed using a similar technique as the one used by \mycite{{gander_analysis_2007}}.
%
%
\begin{itemize}
\item The linear bound in equation \eqref{linear_bound} is obtained by writing 
$\| H_x^k \|_{\infty} \leq \frac{1}{(1-|G|)^k}$ and 
$ \| C_x \|_{\infty} \leq (F-G)$ and 
$\| M^{-1}_x \|_{\infty} \leq \frac{1}{1-|G|}$ (see \mycite[proof of theorem 4.9]{gander_analysis_2007}). 

\sloppy
\item The superlinear bound in equation \eqref{superlinear_bound} is obtained by writing
$\| H^k_x \|_{\infty} 
\leq 
\binom{N-1}{k}$ and 
$ \| C_x \|_{\infty} \leq (F-G)$ and
$
\| M^{-1}_x \|_{\infty} 
=
\frac{1 -|G|^N}{1-|G|}
$
(see \mycite[lemma 4.4]{gander_analysis_2007}). 
\end{itemize}
\end{proof}

The linear and superlinear bounds \eqref{linear_bound} and \eqref{superlinear_bound} capture two effects.
For both of them, the first term equals the convergence bound from \mycite{gander_analysis_2007}, given in lemma \ref{theorems_from_Gander_Vandewalle}. It arises from a difference in $F$ and $G$, for instance due to different time-steppers for the multiscale and averaged ODEs. 
The second term captures the effect of the convergence of the fast variable on the convergence of the slow variable. 
Each of the summands in the second term is a product of a dampening factor (determined by $F$ and $G$) and the error on the fast variable at previous iterations.

In practice, the quantities $e^{x,k}_{\mathrm{max}}$ and $e^{y,k}_{\mathrm{max}}$ are unknown \textit{a priori}. 
$e^{x,k}_{\mathrm{max}}$ is also unknown when the classical bounds from \mycite{gander_analysis_2007} are used.
For the linear ODE that we consider here, one can estimate
$e^{y,k}_{\mathrm{max}} \leq |y_0|$.

\subsection{Numerical verification of the obtained convergence bounds}
\label{convergence_bound_numerical_experiment}

Our Julia implementation for the numerical simulations is available in \cite{ignace_software_paper_ANZIAM}.

\paragraph{Numerical validation of convergence bound.}
We use $\alpha = -1$. 
As an initial condition we use $x_0 = y_0 = 1$.
We consider a time interval $[0, \, T] = [0, \, 2]$, and $N = 20$ subintervals.
We let the parameter $\beta$ vary, as $\beta \in [0, \, 10^{-4}, \, 10^{-2}, \, 10^{-1}, \, 1, \, 2]$.
The fine solver is the exact solution of the multiscale ODE in equation \eqref{definition_mathcal_A}.
The coarse solver is the forward Euler method on the reduced ODE in equation \eqref{reduced_ODE} with time-step $\delta t = 0.1$.

In figure \ref{figure_new_bound_lin_and_suplin}, we display the numerically obtained error $e^{x,k} = \max_{1 \leq n \leq N} |x^k_n - x_n|$, as well as the minimum of the linear and superlinear bounds from lemma \ref{lemma_linear_superlinear_bound}.
We also show the bound based on generating functions in equation \eqref{equation_nontight_bound_result}, where we used the infinity norm for $\norm{\cdot}$.
For this simulation, the decay rate of the fast variable equals $\delta = -5$.

\begin{figure}[H]
\centering
\includegraphics[width=0.8\textwidth]{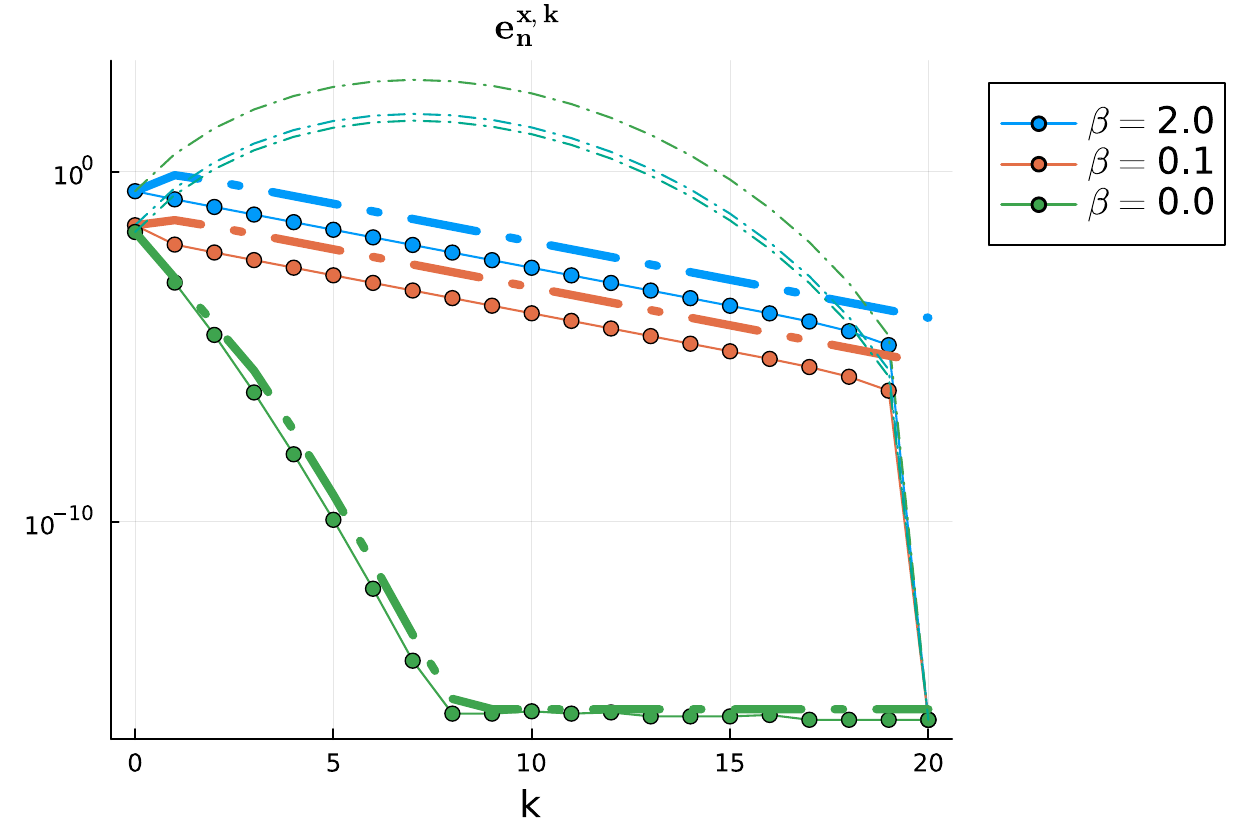}
\caption{
Convergence of micro-macro Parareal for the multiscale ODE test problem, with forward Euler as coarse solver. 
The decay rate of the fast variable equals $\delta = -5$.
Connected circles: numerical simulation, 
thin dashed line: upper bound from equation \eqref{equation_nontight_bound_result}, 
thick dash-dotted line: minimum of linear and superlinear bounds from equations \eqref{linear_bound} and \eqref{superlinear_bound}. 
Each color corresponds to another value of $\beta$.}
\label{figure_new_bound_lin_and_suplin}
\end{figure}

From figure \ref{figure_new_bound_lin_and_suplin} we observe that the new bounds from lemma \ref{lemma_linear_superlinear_bound} effectively capture the convergence, and that these are more tight than the bound \eqref{equation_nontight_bound_result}.
If $\beta = 0$, the slow variable is not influenced by the fast variable; in that case the convergence of the slow variable can be bounded using the classical upper bounds from \mycite{gander_analysis_2007} (see lemma \ref{theorems_from_Gander_Vandewalle}): the convergence speed is solely determined by the time-stepping error of the coarse propagator.
As $|\beta|$ increases, the influence of the convergence of the fast variable (which is linear, see lemma \ref{lemma_error_fast_variable}) on the error on the slow variable increases.

In figure \ref{figure_new_bound_lin_and_suplin_new_exponential}, we study a similar situation but now we artifically increase the error in the coarse solver, by choosing 
$G = e^{\bar \alpha \Delta t}$ with $\bar \alpha = 2 \alpha$, and by putting $\delta =-10$.

\begin{figure}[H]
\centering
\includegraphics[width=0.8\textwidth]{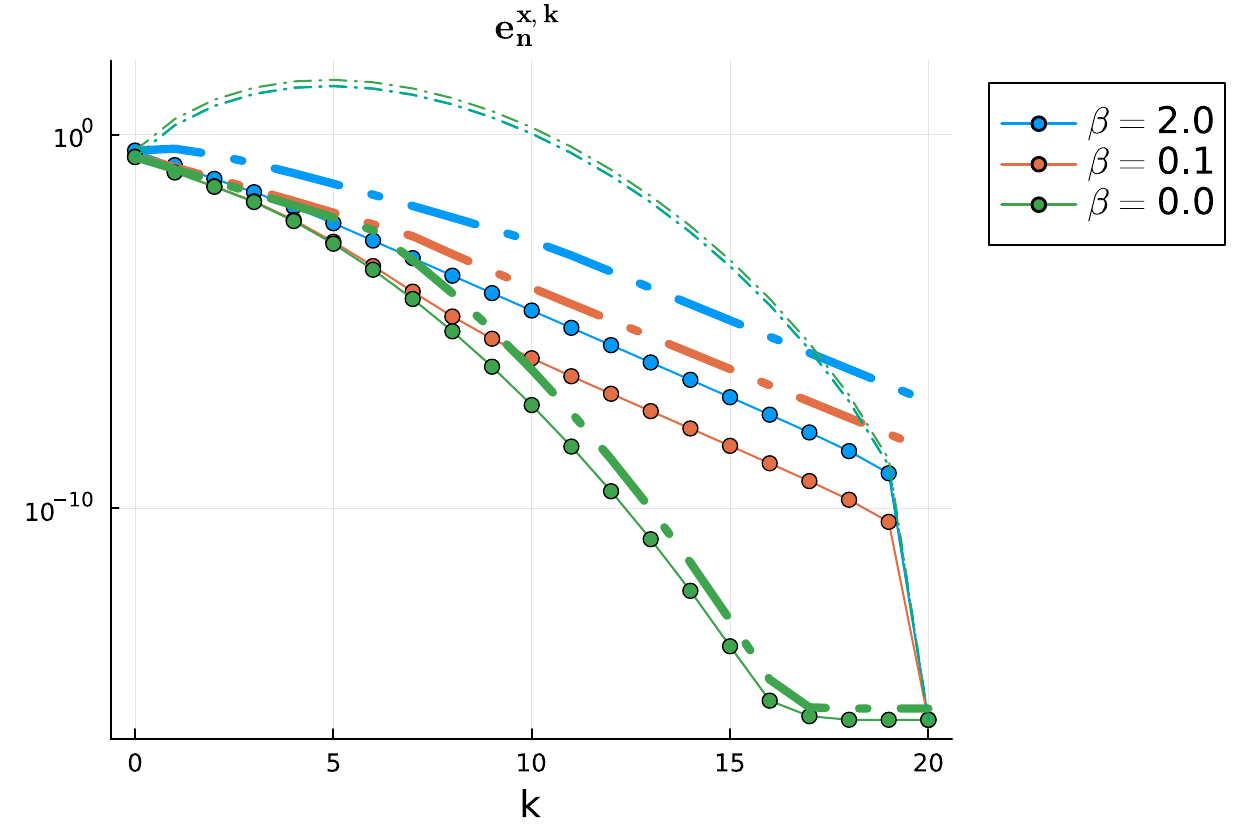}
\caption{
Convergence of micro-macro Parareal for the multiscale ODE test problem, with as coarse solver the exact solution to a modified problem with $\bar \alpha = 2 \alpha$ and $\delta =-10$.}
\label{figure_new_bound_lin_and_suplin_new_exponential}
\end{figure}

From figure \ref{figure_new_bound_lin_and_suplin_new_exponential} we observe the following.
If $\beta =0$, the error in the slow variable is solely due to the model error in the coarse solver.
As $|\beta|$ increases, initially the error decays much like $\beta=0$, but after the first few iterations a linear convergence regime sets in.
In this linear regime, the error is dominated by the convergence of the fast variable, which lemma \ref{lemma_error_fast_variable} has shown to be linear (captured in the second (sum) term of the bounds).
As the number of iterations increases, our bounds have the same slope as the numerically obtained slow errors. 

\paragraph{Avoiding initial slip.}
A proper choice of the initial condition $\bar x_0$ in equation \eqref{reduced_ODE} to avoid initial slip, following \mycite[Appendix A]{van_kampen_elimination_1985} is
\begin{equation}
\bar x_0 = x_0 - y_0 \frac{\beta}{\delta - \alpha}.
\label{better_initial_condition}
\end{equation}
This choice does affect the operators from definition \ref{definition_micro_macro_Parareal_ODEs}.
Indeed, if the restriction operator is chosen as $\mathcal R \left( (x,y) \right) = x - y \frac{\beta}{\delta-\alpha}$ and if the matching operator is the one from definition \ref{definition_micro_macro_Parareal_ODEs}, then it does not hold that $\mathcal M ( \mathcal R (u), u) = u$ for any $u$.
This would lead to a potential violation of the finite termination property (see property \ref{properties_micro_macro_Parareal}).

Alternatively, we can shift the effect of equation \eqref{better_initial_condition} into the coarse propagator, which then acts on the full micro state variable.
If
\begin{equation}
\hat{\mathcal B} 
= \left[ \begin{matrix}
e^{\bar \alpha \Delta t} & -\frac{\beta }{\delta -\alpha}e^{\bar \alpha \Delta t} \\
0 & 0
\end{matrix} \right],
\label{definition_mathcal_B_hat}
\end{equation}
then the following classical Parareal iteration describes the effect the modified initial condition from equation \eqref{better_initial_condition}: 
$$
\begin{aligned} 
u^0_{n+1} &= \hat{\mathcal B} u^0_n
\\
u^{k+1}_{n+1} 
&=
(\mathcal A - \hat{\mathcal B})
u^{k}_{n}
+ 
\hat{\mathcal{B}} u^{k+1}_n
\end{aligned}
$$.
The proof of this equality is similar to the proof of lemma \ref{solution_micro_macro_Parareal}. 
In micro-macro Parareal, this corresponds to the following choices: the micro variable $u = ( x, \, y )$ and the macro variable $U = (x, \, y)$, and 
\begin{itemize}
\item a fine propagator $\mathcal F: 
\mathbb R^2 \rightarrow \mathbb R^2: 
\mathcal F(u) := \mathcal A u$ (see the stiff model \eqref{multiscale_ODE}) 

\item a coarse propagator $\mathcal C: 
\mathbb R^2 \rightarrow \mathbb R^2: 
\mathcal C(U) := \hat{\mathcal B} U$,

\item a restriction operator
$ \mathcal R: 
\mathbb R^2 \rightarrow \mathbb R^2: 
\mathcal R (u)
:= 
u$, 

\item a matching operator  
$ \mathcal M: 
(\mathbb R^2, \mathbb R^2) \rightarrow \mathbb R^2:
\mathcal M((X,Y), (x,y)) 
:=
(X,y)$

\item a lifting operator 
$ \mathcal L: 
\mathbb R^2 \rightarrow \mathbb R^2:
\mathcal L((x,y)) 
:= 
(x, \, 0) $.
\end{itemize}

In figure \ref{figure_initial_slip} we can see that this proper choice of initial condition moderately improves the error. A theoretical analysis of this effect, however, is beyond the scope of this work.

\begin{figure}[h]
\centering
\includegraphics[width=0.8\textwidth]{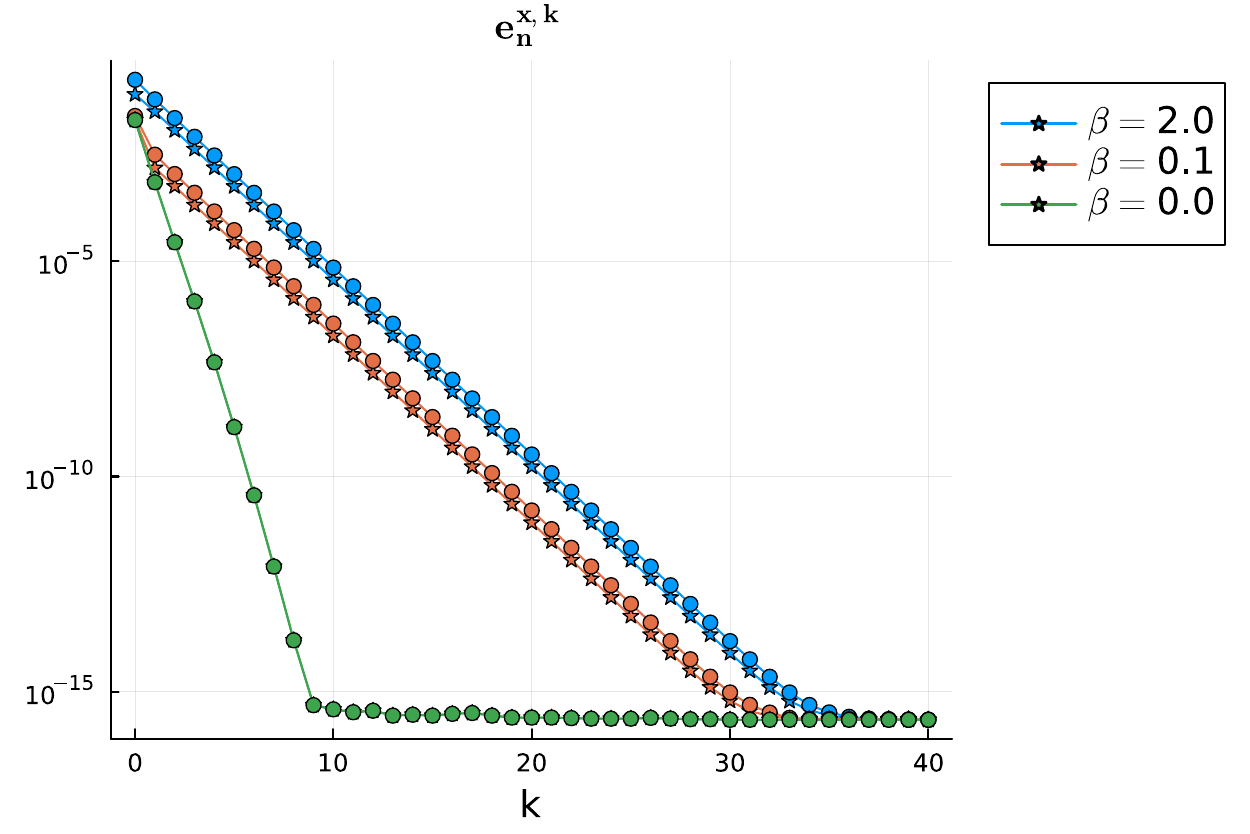}
\caption{
Avoiding initial slip: effect on the convergence.
$\alpha = -1$ and $\delta = -10$.
The star-shaped simulations are obtained with the modified procedure (using $\hat{\mathcal B}$), the filled circles use the operators from definition \ref{definition_micro_macro_Parareal_ODEs}.
}
\label{figure_initial_slip}
\end{figure}

\section{Monte Carlo-moments Parareal: extension to multidimensional SDEs}
\label{section_micro_macro_SDE}
In \cite{bossuyt_monte_carlomoments_2023}, we proposed the Monte Carlo-moments (MC-moments) Parareal method, a micro-macro Parareal algorithm for the simulation of scalar McKean-Vlasov stochastic differential equations.
In this section we extend that method to multidimensional SDEs. 

\subsection{McKean-Vlasov stochastic differential equations}
\sloppy
McKean-Vlasov SDEs model the dynamics over time $t$ of an ensemble of $P$ (coupled) particles $x^{(p)}  \in \mathbb{R}^d$, $p=1 \hdots P$.
The system is defined for some initial distribution of particles $p_0$.
Each particle $x^{(p)}$ is subject to noise from the Wiener process $W(t,\omega^{(p)}$) for noise realisations parametrised by $\omega^{(p)}$.
For notation, these particles are collected in the state variable
$\bar x = \left( x^{(1)}, \hdots, x^{(P)} \right)$.

We consider a special class of Mckean-Vlasov SDEs, where the mean-field interaction is the expected value of a function of the particles (for a more general treatment, see, e.g., \mycite{snitzman_1991}).
Let $\psi$ be some function of the collective state variable, and $\Lambda(t) = \mathbb{E}_P[\psi(\bar x(t))]$, in which the expectation $\mathbb{E}_P$ is taken over the particle ensemble $\bar x$, where each particle is one realisation of the stochastic process.
For instance, if $\psi$ is the identity, then one is interested in just the mean of the particle ensemble. 
Later we also use the operator $\mathrm{Cov}_P[\bar x]$, which computes the covariance matrix of the ensemble $\bar x$. 

This is the McKean-Vlasov SDE that we consider:
\begin{equation}
\begin{aligned}
dx^{(p)} &= 
a \left( x^{(p)}, \Lambda(t), t \right) dt
+ 
b \left(
x^{(p)}, \Lambda(t), t \right) dW^{(p)}, \\
\bar x(0) &\sim p_0, \\ 
\end{aligned}
\label{general_equation}
\end{equation}
in which the coefficient $a \in \mathbb{R}^d$ is a drift coefficient, and $b \in \mathbb{R}^{d \times m}$ is a diffusion coefficient.
We use the It\^{o} interpretation of the SDE \eqref{general_equation}.

\paragraph{Assumptions.} 
We only consider SDEs whose particle distribution is unimodal at all times $t \geq 0$.
For more information about loosening these assumptions, see \mycite{bossuyt_monte_carlomoments_2023}.

\paragraph{Simulation.}
The Monte-Carlo simulation of these equations is computationally expensive because (i) for accurate results, a lot of particles (samples) are required, and (ii) in each time-step, the mean of the particles in $\Lambda(t)$ needs to be computed.

\subsection{Moment ODEs as approximation for multidimensional SDEs}
In this section, we aim to obtain a cheap approximation to the McKean-Vlasov SDE in equation \eqref{general_equation}.

We first write a general multivariate Taylor series of the drift and diffusion coefficients for a particle $x^{(p)}$ around the mean value $\mathbb{E}_P[\bar x]$:
\begin{equation}
a(x^{(p)}, \Lambda,t) 
\approx 
a(\mathbb{E}_P[\bar x], \Lambda, t) + A_1(\mathbb{E}_P[\bar x], \Lambda,t) (x^{(p)}-\mathbb{E}_P[\bar x]) 
\end{equation}
\begin{equation}
b(x^{(p)}, \Lambda,t) 
\approx 
b(\mathbb{E}_P[\bar x], \Lambda,t) + B_1(\mathbb{E}_P[\bar x], \Lambda,t) (x^{(p)}-\mathbb{E}_P[\bar x]).
\end{equation}
Here, $A_1$ and $B_1 \in \mathbb{R}^{d \times d}$  are the Jacobian matrices of the drift and diffusion coefficients $a$ and $b$.

This expansion, however, is not practically useful, because it assumes that $\Lambda$ is available. A more practical expansion is constructed by approximating $\Lambda = \mathbb{E}_P[\psi(\bar x)]$ with $\Lambda_{\mathbb{E}_P} = \psi(\mathbb{E}_P[\bar x])$ (the approximation consists in changing the position of the expected value operator).

We now define ODEs for describing the evolution the variables $M$ and $\Sigma$, which aim to closely approximate the mean and the covariance matrix of the particle ensemble. 
For background theory on these approximations for classical SDEs, see for instance \mycite{rodriguez_statistical_1996} and \mycite{arnold_stochastic_1974}.
The variable $M$ approximates the mean $\mathbb{E}_P[\bar x] \in \mathbb{R}^d$ as
\begin{equation}
\begin{aligned}
\frac{dM}{dt} &= a(M,\psi(M)) + \frac{1}{2} Q \operatorname{Vec} (\Sigma), \qquad 
M(0) &= \mathbb{E}_P[\bar x(0)],
\end{aligned}
\label{moment_model_mean}
\end{equation}
where $Q$ is a matrix whose $j$-th row contains a flattened Hessian matrix of the $j$-the component of $a$ and where $\operatorname{Vec}$ is the vectorisation (flattening) operator.
The evolution of an approximation $\Sigma$ to the covariance matrix $\mathrm{Cov}[\bar x] \in \mathbb{R}^{d \times d}$ can be described with a matrix differential equation:
\begin{equation}
\begin{aligned}
\frac{d\Sigma}{dt} 
&= 
A_1(M,\psi(M),t)\Sigma 
+ \Sigma A_1(M,\psi(M),t)^T 
+ B_1 \Sigma B_1^T \\
& \quad + b(M,\psi(M),t) b(M,\psi(M),t)^{T} \\
\Sigma(0) &= \mathrm{Cov}_P[\bar x(0)].
\end{aligned}
\label{moment_model_covariance}
\end{equation}




\sloppy
\subsection{Coupling operators for multidmensional MC-moments Parareal}
\label{section_coupling_MC_moments}

\begin{definition}[MC-moments Parareal for SDEs of any dimension]
\label{definition_MC_moments_Parareal}
The micro variable $u$ is an ensemble of particles $u = \bar x$, and the macro variable $U$ contains its first moments (mean and covariance matrix) $U = [\mathbb{E}_P[u], \, \mathrm{Cov}_P[u]]$. We choose

\begin{itemize}
\item a fine propagator $\mathcal F$: a Monte Carlo simulation of the SDE \eqref{general_equation}

\item a coarse propagator $\mathcal C$: (a numerical simulation of) the moment model \eqref{moment_model_mean}-\eqref{moment_model_covariance}.

\item a restriction operator 
$\mathcal R( \bar x) 
:= 
[\mathbb{E}_P[\bar x], \, \mathrm{Cov}_P[\bar x]]$.

\item a matching operator $\mathcal M$ that takes as input a particle ensemble $\bar x$ and a pair of desired mean $M$ and covariance matrix $\Sigma$.
Let us assume that both $\mathrm{Cov}_P[\bar x]$ and $\Sigma$ are positive semidefinite (PSD).
Now define the matrices $V \in \mathbb R^{d \times d}$ and $Q \in \mathbb R^{d \times d}$ via the  Cholesky decompositions $\Sigma = V V^T$ and $\mathrm{Cov}_P[\bar x] = Q Q^T$. 

\begin{itemize}
\item If $Q$ is invertible, then we define $A = V Q^{-1}$, and 
\begin{equation}
\mathcal{M}([M, \, \Sigma], \bar x) 
:= 
A (\bar x - \mathbb{E}_P[\bar x]) + M
\label{definition_matching_operator}
\end{equation}

\item If $Q$ is not invertible, then the particles $\bar x$ are resampled from a standard normal distribution.
\end{itemize} 
 
\item a lifting operator $\mathcal L$, used in the zeroth micro-macro Parareal iteration is choosen to be matching with respect to the initial condition of the SDE $\mathcal{L}(U) := \mathcal{M}(U, \bar x(0) )$.
\end{itemize}
\end{definition}

If the drift and diffusion coefficients are affine functions in their first two arguments $M$ and $\Lambda = \phi(M)$, and if $\phi$ is a linear function as well, then the moment equations \eqref{moment_model_mean}-\eqref{moment_model_covariance} are an exact description of $\mathbb{E}_P[\bar x]$ and $\mathrm{Cov}_P[ \bar x]$. 
This fact indicates that the choice of $\mathcal R$ and $\mathcal C$ is sensible.


We now present a lemma that is important in the light of the exactness property of micro-macro Parareal (see property \ref{properties_micro_macro_Parareal}).
This lemma is proven in appendix \ref{proof_property_matching}.

\begin{lemma}[Micro-macro consistency of the choosen matching operator]
\label{lemma_matching_operator}
Let $\bar x$ be an ensemble of particles such that $\mathrm{Cov}_P[\bar x]$ is symmetric positive definite.
Let $U = \left[ \begin{matrix} M, \Sigma \end{matrix} \right]$ be a macro variable, with $\Sigma$ PSD.
Then the matching operator \eqref{definition_matching_operator} satisfies $\mathcal{R}\left( \mathcal M( U, u) \right) = U$ and $\mathcal{M} \left( \mathcal R (u), u \right) = u$.
\end{lemma}
\noindent


\subsection{Numerical experiment}

We illustrate the generalised MC-moments Parareal method from definition \ref{definition_MC_moments_Parareal} for a classical SDE without mean-field interaction
from the book by 
\mycite[Exercise 21.6]{roberts_model_2014}:
\begin{equation}
\begin{aligned}
dx &= f(x,y) dt \\
dy &= g(x,y) dt + \sigma dW
\end{aligned}
\qquad
\qquad 
\left[ 
\begin{matrix}
x(0) \\ y(0)
\end{matrix}
\right] 
= 
\left[ 
\begin{matrix}
1 \\ 1
\end{matrix}
\right],
\label{SDE_Roberts_definition}
\end{equation}
with $f(x,y) = \alpha x - x y$ and $g(x,y) = -y + x^2$.
Our aim is to simulate the SDE over the time interval $[0,  \, 10]$.
For this choice of initial codition, the solution is unimodal in the whole time domain. 
Since the noise is additive, the It\^{o} interpretation of \eqref{moment_model_Roberts} is the same as its Stratonovich interpretation.
For the stochastic simulation of the SDE, we use the Euler-Maruyama method (see e.g. \mycite{kloeden_platen_1999}) with a time step of $\Delta t = 0.02$ on a particle ensemble $\bar x$ with $10^5$ particles.

For the SDE \eqref{SDE_Roberts_definition}, the moment model, defined in \eqref{moment_model_mean}-\eqref{moment_model_covariance}, can be written as
\begin{equation}
\begin{aligned}
\frac{dM_x}{dt} &= f(M_x,M_y) + f_{xy}(M_x, M_y) C_{xy} \\
\frac{dM_y}{dt} &= g(M_x, M_y) C_{xx} + \frac{1}{2} g_{xx}(M_x, M_y)C_{xx}  \\
\frac{dC_{xx}}{dt} &= 2f_x(M_x,M_y)C_{xx} + 2f_y(M_x,M_y)C_{xy} \\
\frac{dC_{xy}}{dt} &= f_y(M_x,M_y)C_{yy} + (f_x(M_x,M_y) + g_y(M_x,M_y))C_{xy} + g_x(M_x,M_y)C_{xx} \\
\frac{dC_{yy}}{dt} &= 2g_y(M_x,M_y)C_{yy} + 2g_x(M_x,M_y)C_{xy} + \sigma^2,
\end{aligned}
\label{moment_model_Roberts}
\end{equation}
with
\begin{equation}
\left[
\begin{matrix}
M_x(0) \\
M_y(0) 
\end{matrix}
\right]
=  
\left[
\begin{matrix}
1 \\ 1
\end{matrix}
\right]
\qquad 
\qquad
\left[
\begin{matrix}
\Sigma_{xx}(0) & \Sigma_{xy}(0) \\
\cdot & \Sigma_{yy}(0)
\end{matrix}
\right]
=  
\left[
\begin{matrix}
0 & 0 \\
\cdot & 0 
\end{matrix}
\right].
\end{equation} 
and with $f_x = \alpha - y$, $f_y = -x$, $f_{xy}=-1$, $g_x = 2x$, $g_y = -1$ and $g_{xx} = 2$. We choose $\alpha=  1$. 
The moment ODE is numerically simulated with the forward Euler method with the same timestep, using the software by \mycite{rackauckas_differentialequationsjl_2017}.

\paragraph{Introductory experiment: Quality of the moment model.}
As a way to inspect the quality of the moment ODE, we plot in figure \ref{figure_Roberts_covar} a numerical solution to the moment ODE together with a numerical solution of the SDE. 
We observe that the quality of the moment model is not always excellent. 
The moment model gets more accurate as the noise level $\sigma$ goes to zero. 
\begin{figure}[H]
\includegraphics[width=0.5\textwidth]{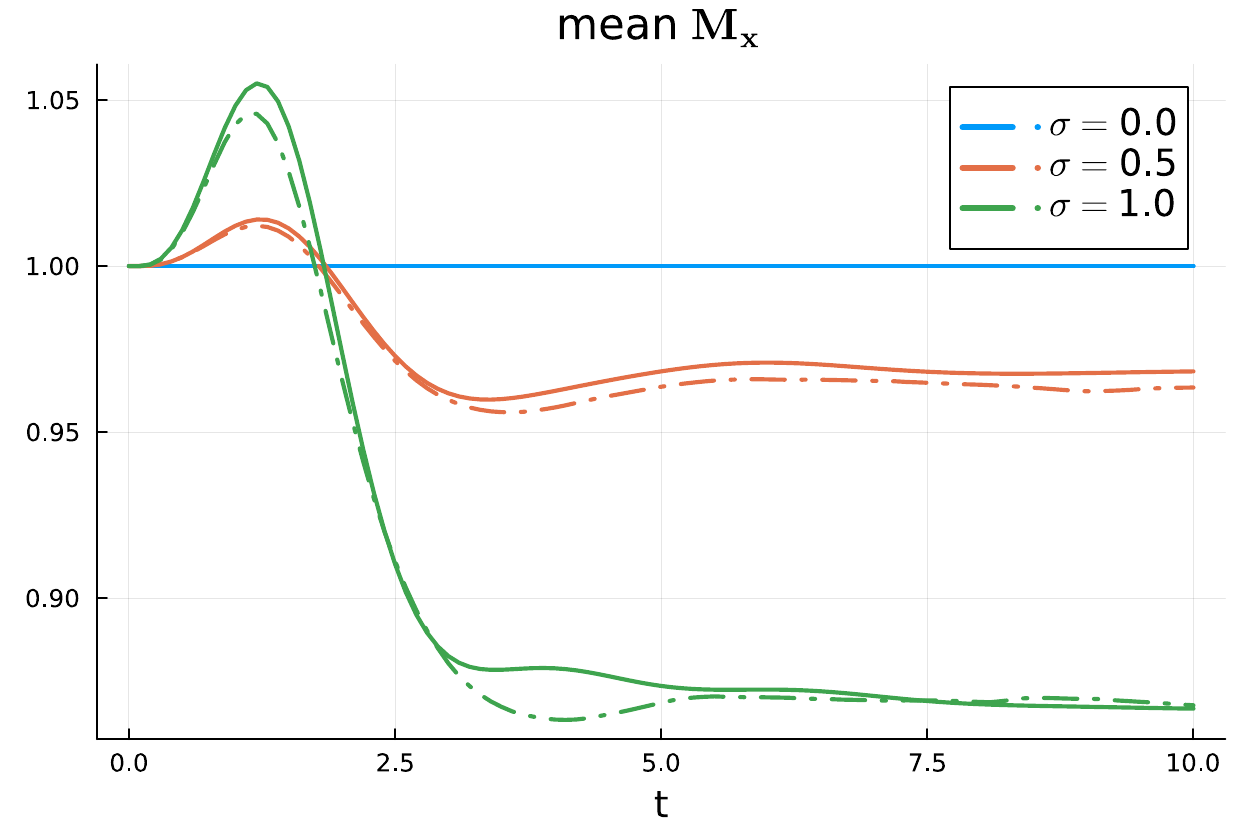}
\includegraphics[width=0.5\textwidth]{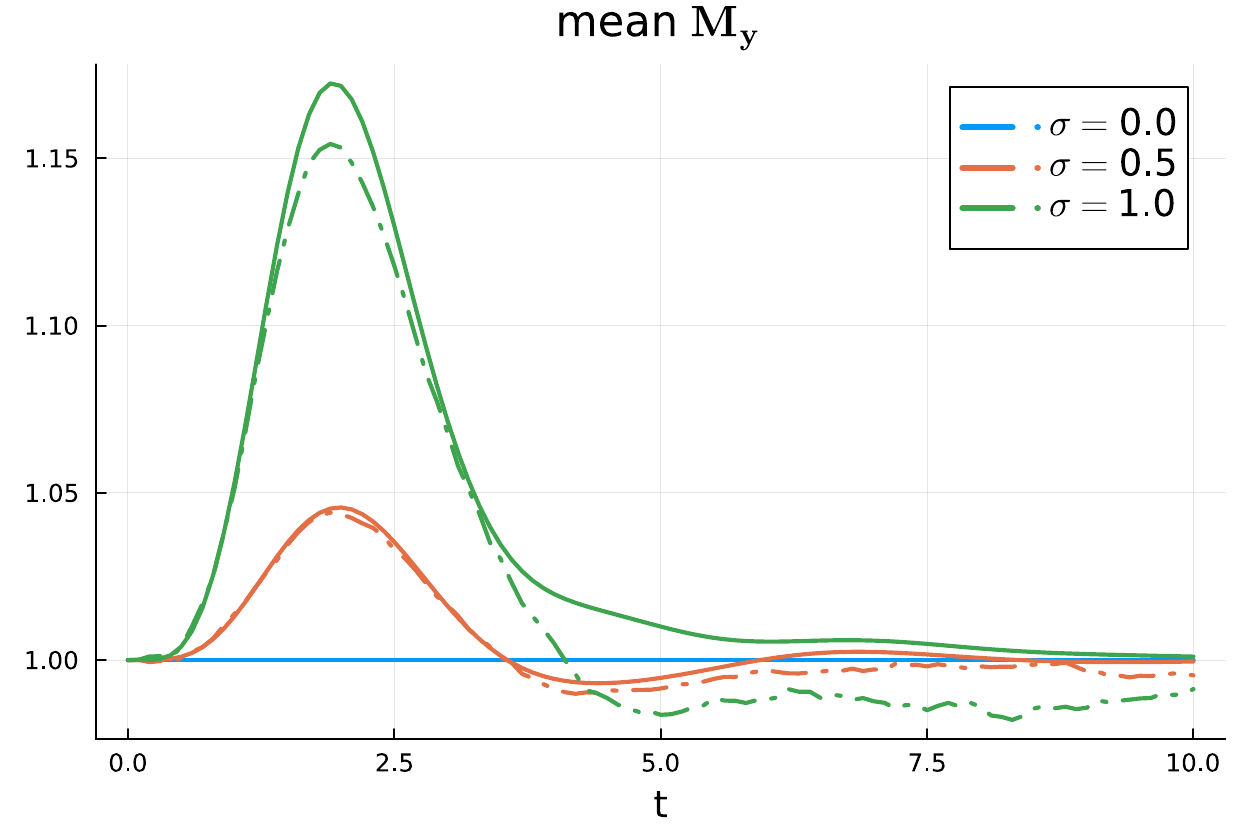}

\includegraphics[width=0.5\textwidth]{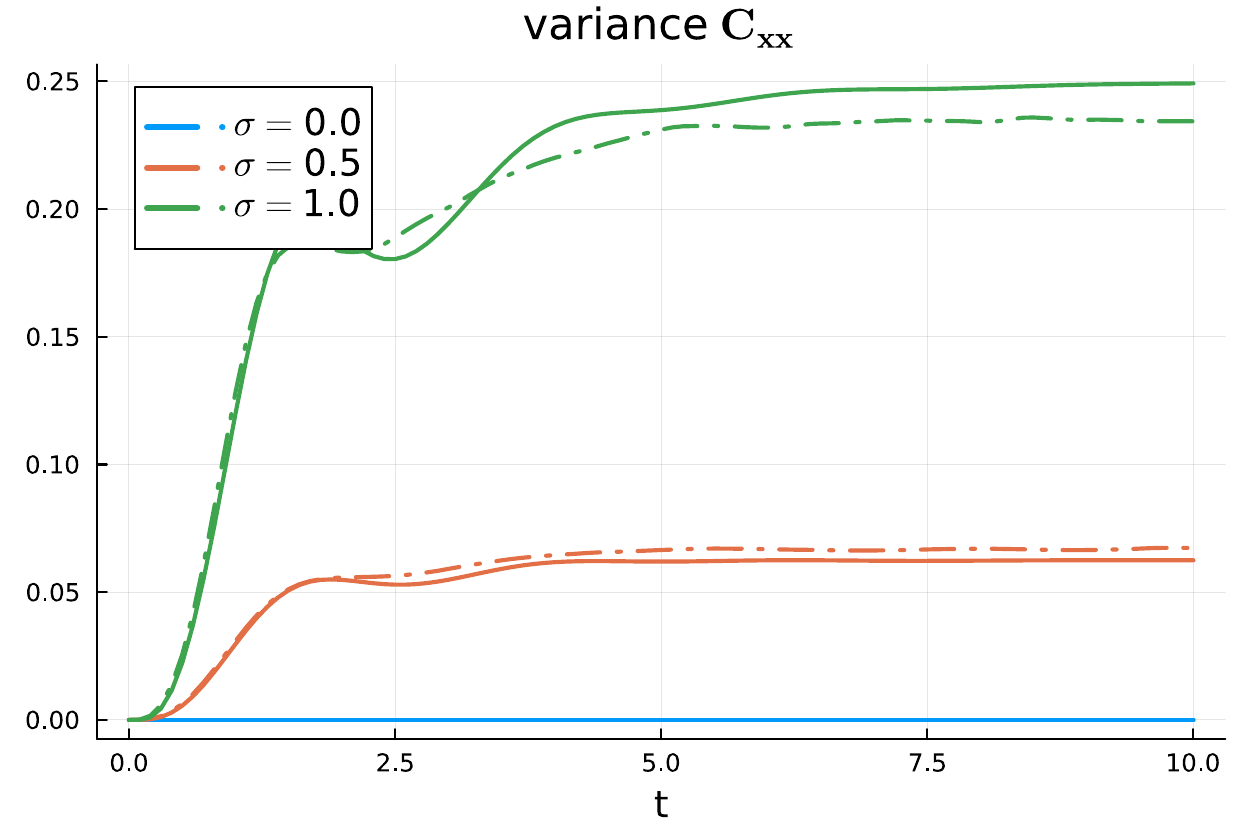}
\includegraphics[width=0.5\textwidth]{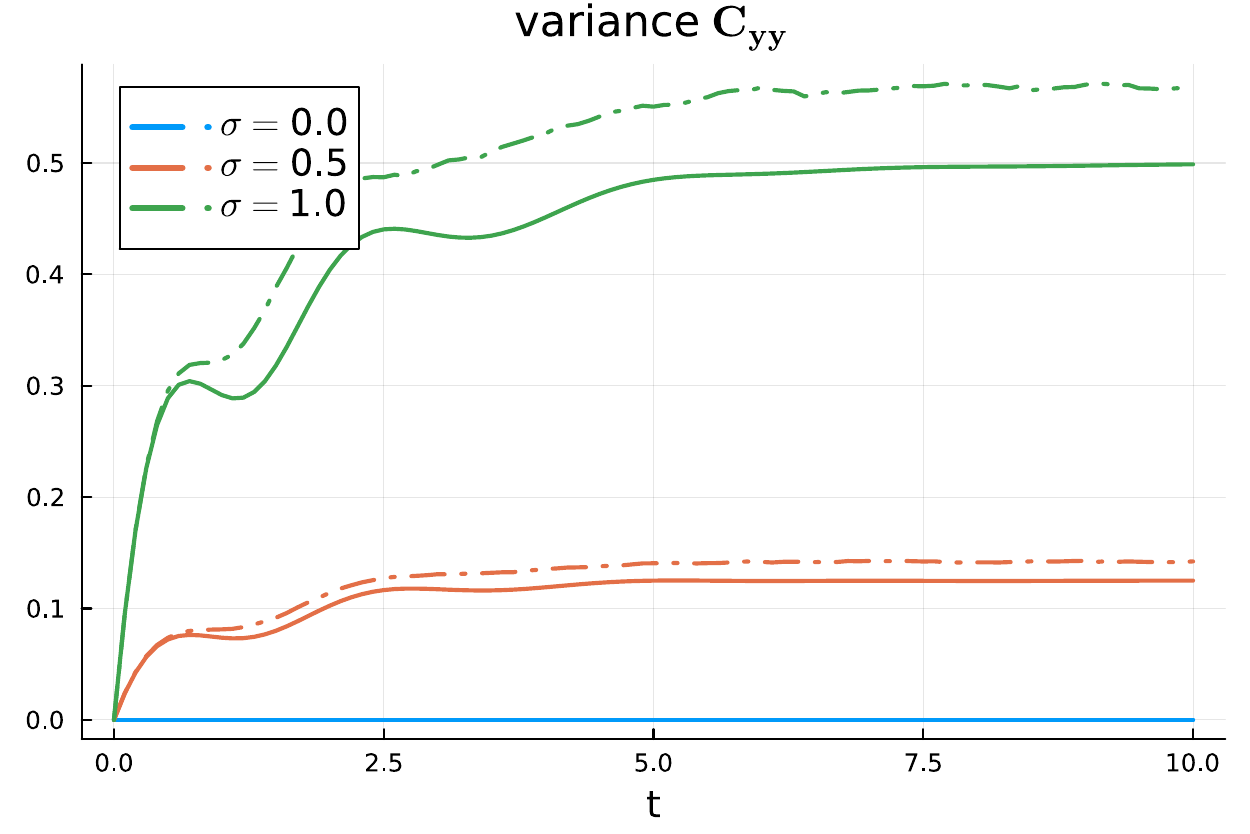}
\includegraphics[width=0.5\textwidth]{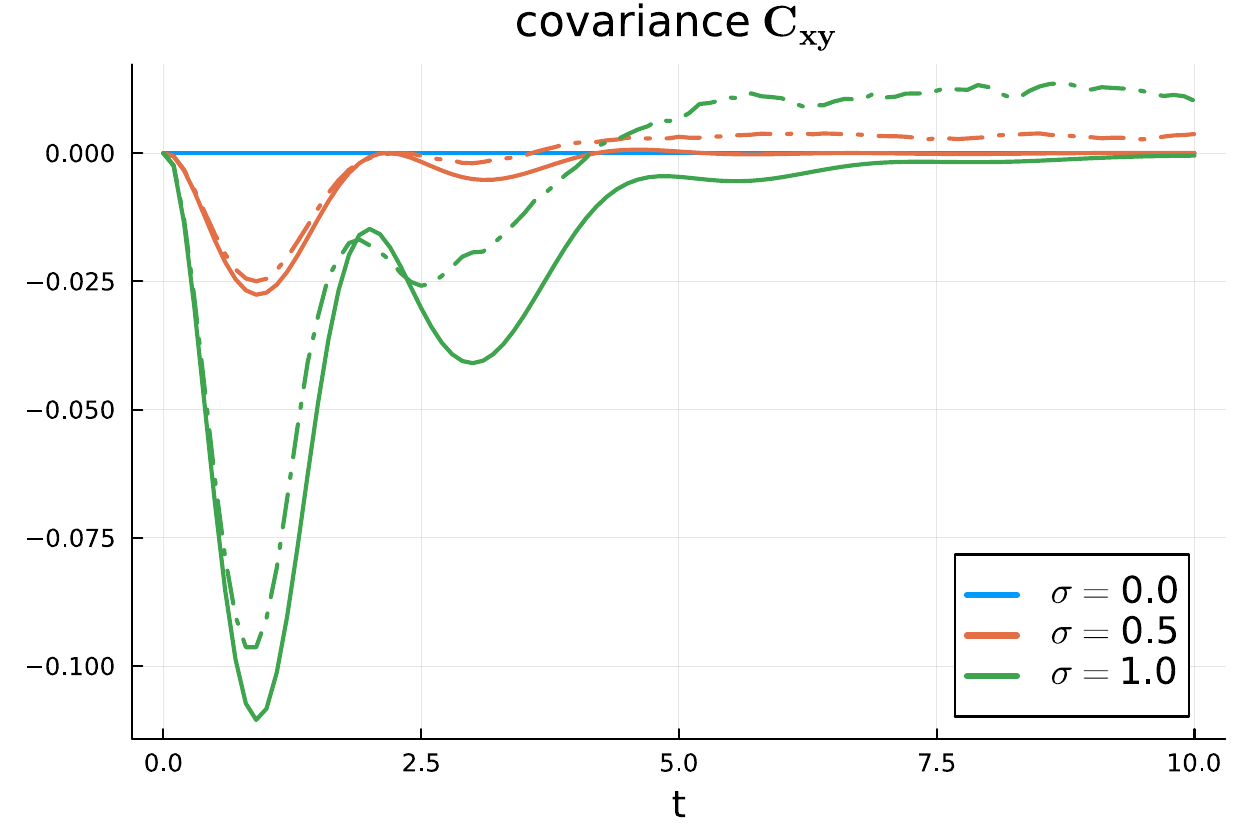}
\caption{Illustration of the moment model \eqref{moment_model_Roberts} as an approximation to the SDE \eqref{SDE_Roberts_definition}. The dashed lines correspond to a stochastic simulation of the SDE, the full lines correspond to the moment model. Different colors indicate different noise levels $\sigma$.}
\label{figure_Roberts_covar}
\end{figure}

\paragraph{Convergence of MC-moments Parareal}
In figure \ref{figure_Roberts_Parareal_convergence} we plot the relative error on each component of the mean and the variance, measured through an infinity-norm over time, as a function of the iteration number. 
The errors are computed using an average of 20 realisations, where in each realisation the fine Parareal solver uses one ensemble $\bar x$ of $P = 10^5$ coupled particles. 
We chose the number of Parareal subintervals $N = K= 10$ and a time interval $[0, 20]$. 
The noise level is set to $\sigma = 0.5$.
In figure \ref{figure_Roberts_covar_Parareal} we plot different Parareal iterates.  
The final iteration (for which $k=K=N$), which is indicated with a thick line, serves as the reference solution, and corresponds to the sequential simulation in figure \ref{figure_Roberts_covar}. 
Both in figures \ref{figure_Roberts_covar_Parareal} and \ref{figure_Roberts_Parareal_convergence} we conclude that convergence takes place. 
In other experiments, however, we have observed that the convergence is parameter dependent.
For instance, changing $\sigma$ and/or $\alpha$ may lead to non-monotonic and slower convergence than the convergence observed in figure \ref{figure_Roberts_Parareal_convergence}.

\begin{figure}[H]
\includegraphics[width=0.96\textwidth,trim={0.6cm 0 0 1.8cm 1.2cm},clip]{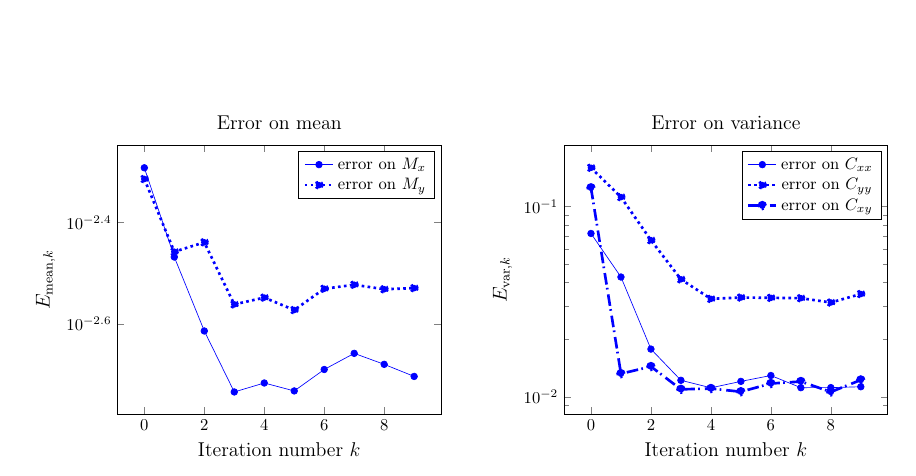}
\caption{Convergence of modified MC-moments Parareal for the multidimensional SDE \eqref{SDE_Roberts_definition}: relative error on mean and variance versus iteration number.}
\label{figure_Roberts_Parareal_convergence}
\end{figure}

\begin{figure}[H]
\includegraphics[width=0.5\textwidth]{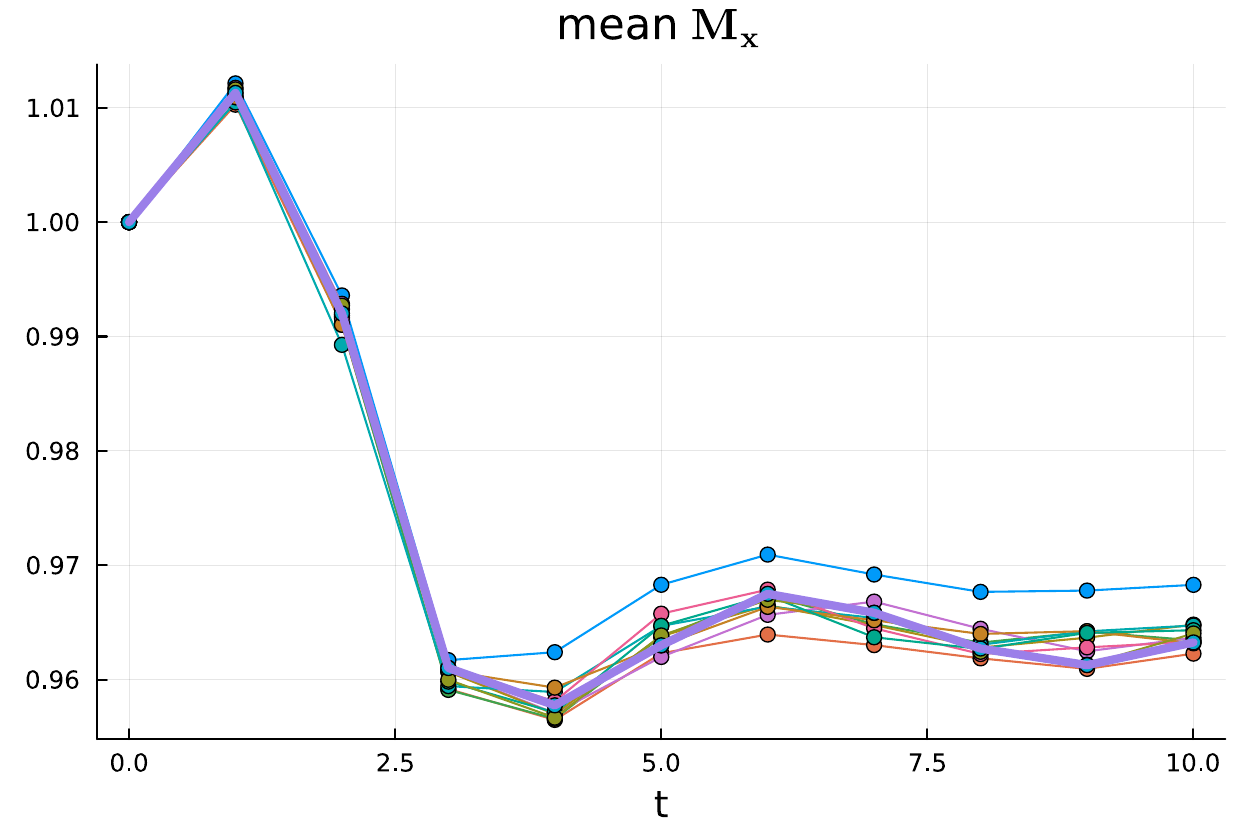}
\includegraphics[width=0.5\textwidth]{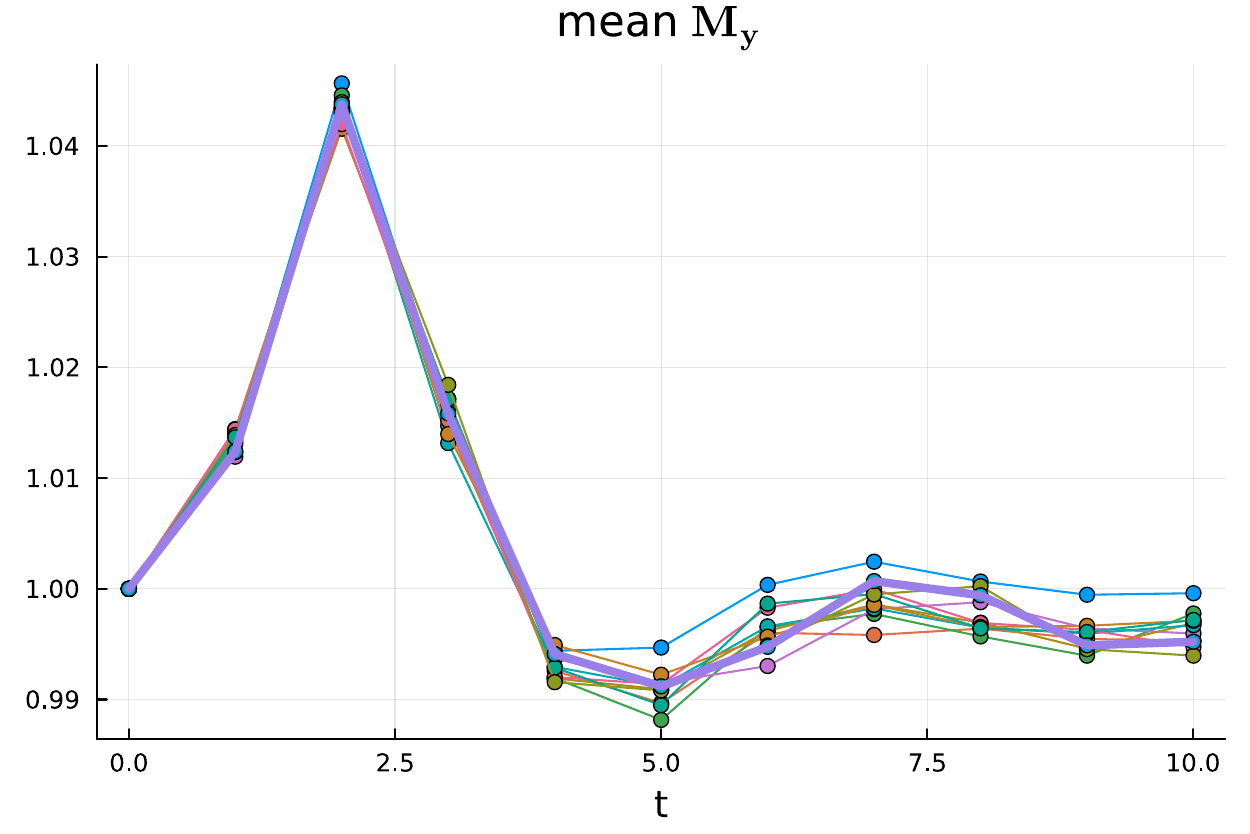}

\includegraphics[width=0.5\textwidth]{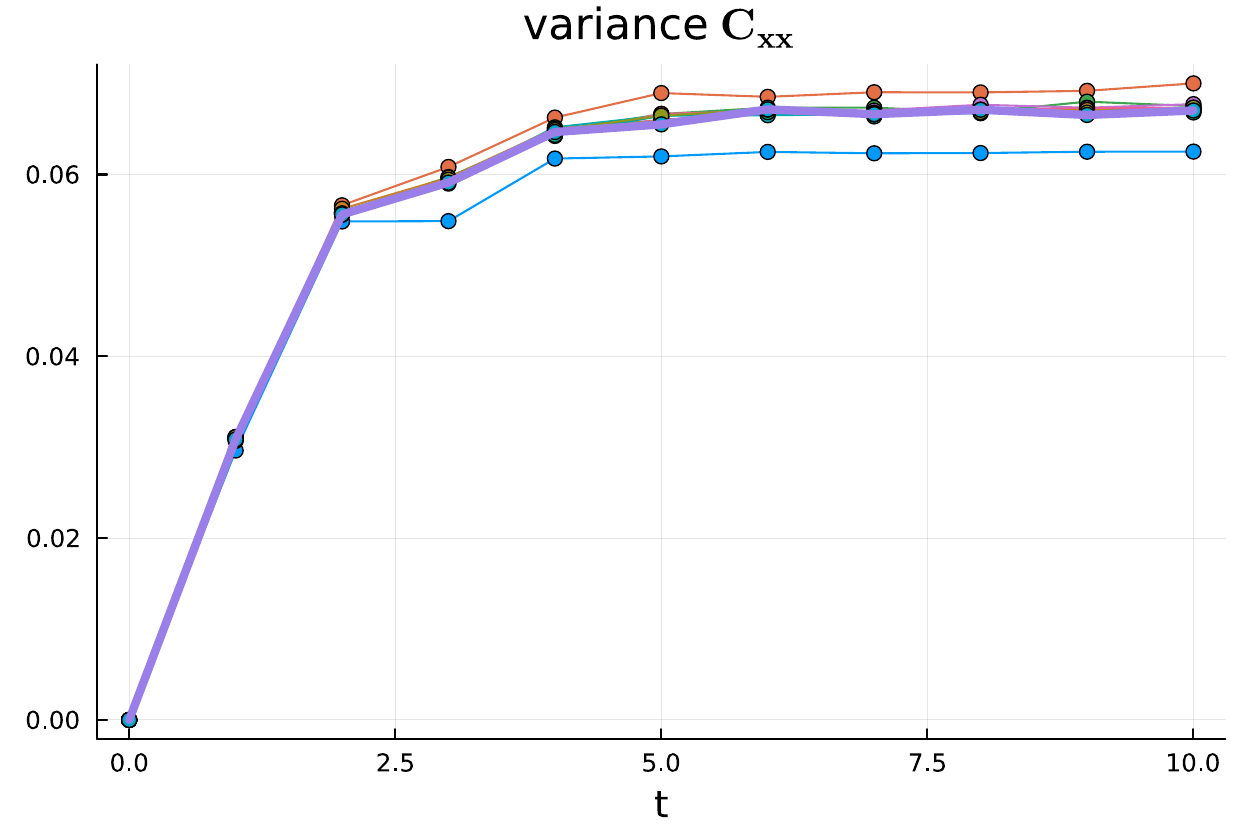}
\includegraphics[width=0.5\textwidth]{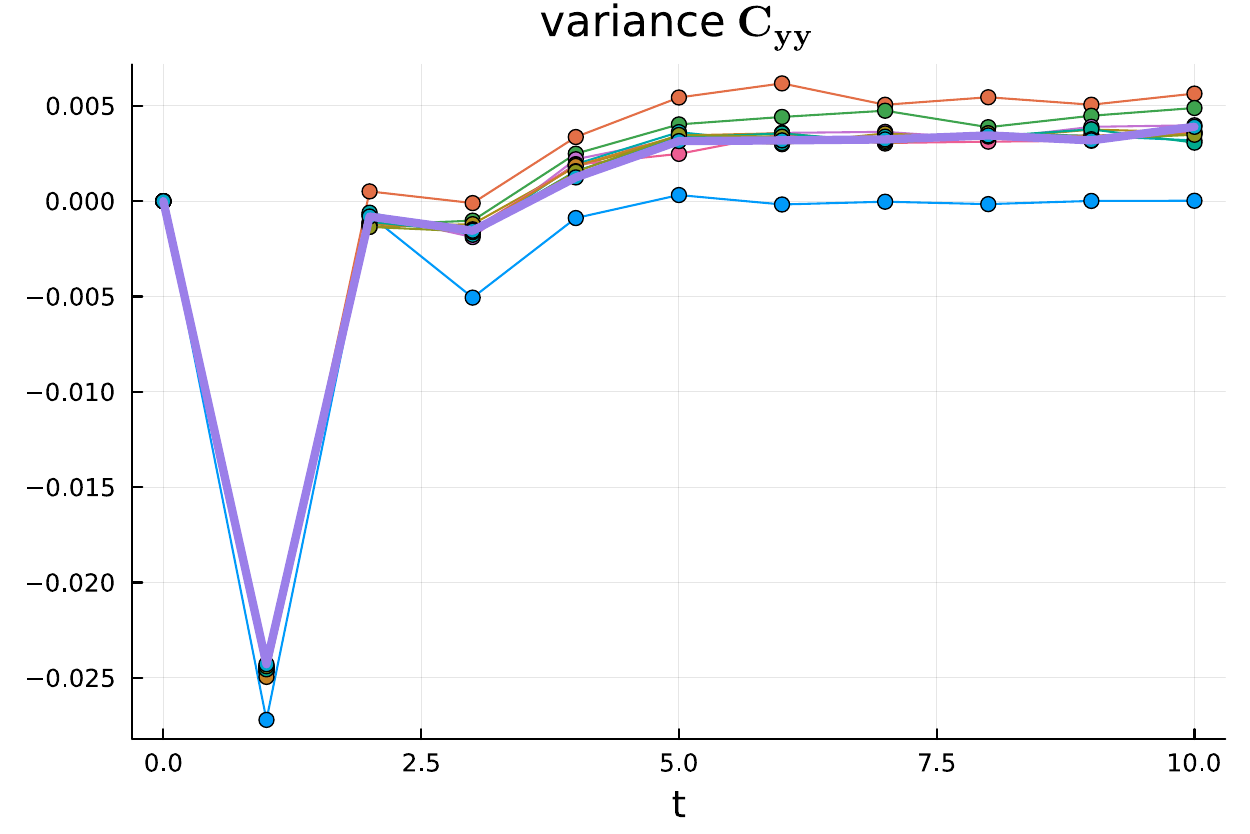}
\includegraphics[width=0.5\textwidth]{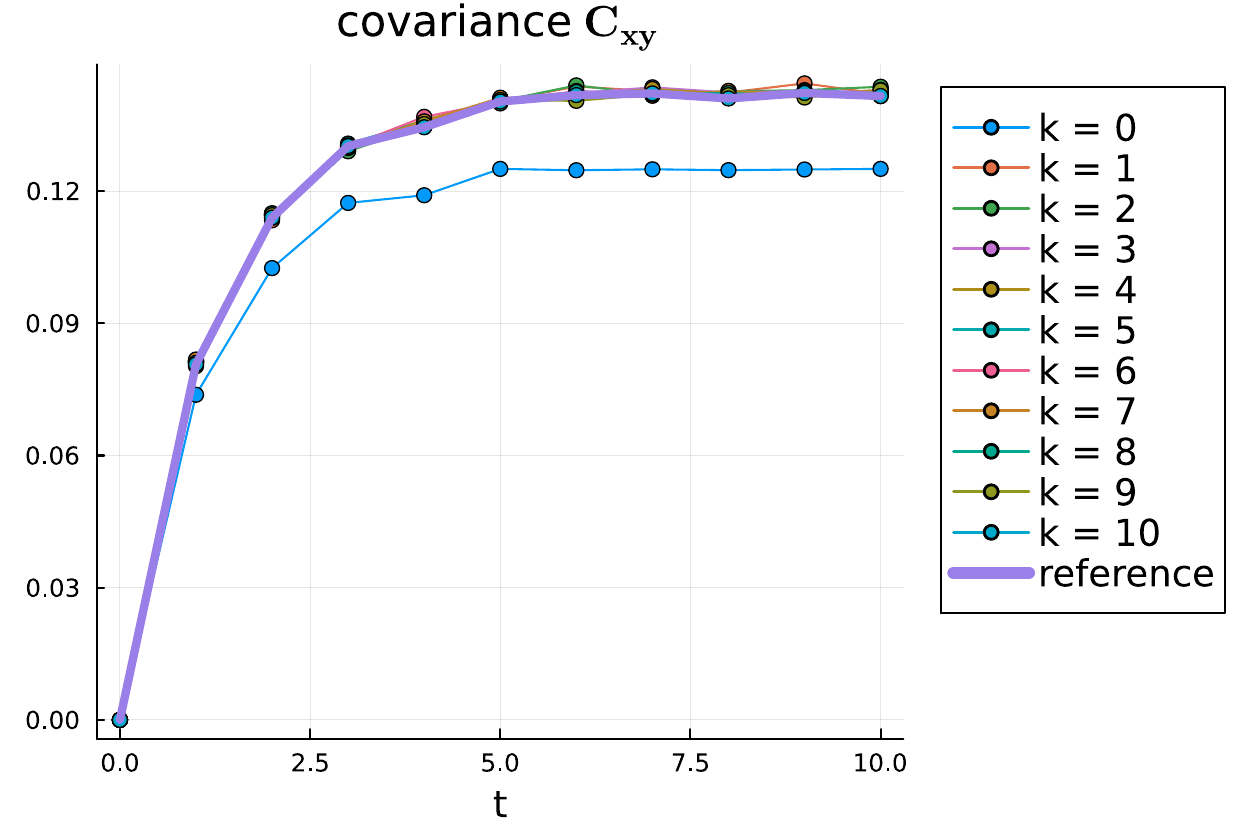}
\caption{Convergence of modified MC-moments Parareal for the multidimensional SDE \eqref{SDE_Roberts_definition}: illustration of the convergence of the moments towards the reference solution. The reference solution is the last iterate of MC-moments Parareal, which corresponds to a sequential simulation with the fine propagator.}
\label{figure_Roberts_covar_Parareal}
\end{figure}

\section{Conclusion}
In section \ref{section_micro_macro_ODE}, we presented a new convergence bound for micro-macro Parareal for a special linear multiscale ODE.
While our work focuses on a specific case of the multiscale ODEs considered by \mycite{Legoll2013}, our bound is a generalisation of an existing bound for scalar ODEs from \mycite{gander_analysis_2007}. 
It contains an extra term that decreases as the decay rate of the fast component in the ODE system increases.
In the future, it would be interesting te develop a convergence bound for a more general class of linear multiscale ODEs, where the bottom left element of $\mathcal A$ does not need to be zero. 
Also, as we considered only $x,y \in \mathbb R$, generalizations to higher dimensions and to more general nonlinear ODEs would be interesting.

In section \ref{section_micro_macro_SDE}, we extended a Parareal algorithm in which the fine solver is a Monte Carlo simulation of the SDE, and the coarse solver is an approximate ODE description of the first two statistical moments of the SDE. 
In the future, it would be interesting to perform tests with the algorithm on more examples, such as true McKean-Vlasov SDEs, and to study the influence of the form of different SDEs on the convergence.
Another path would be to augment the MC-moments Parareal method with a way to exploit other reduced models, such as the normal forms described in \mycite{roberts_model_2014}. 

In summary, we first studied micro-macro Parareal for ODEs, and then we visited a Parareal method for SDEs that exploits ODE descriptions.
This journey brought us first to ODEs, and via SDEs we returned to ODEs.

\paragraph{Acknowledgements.}
We sincerely thank the anonymous reviewers for their careful and thorough work. Their suggestions  have greatly ameliorated this text.

We thank Thibaut Lunet for discussions about convergence of micro-macro Parareal. We thank Thijs Steel for suggesting the Cholesky factorization in the definition of the matching operator \eqref{definition_matching_operator}. We also thank Arne Bouillon for his comments on an earlier version of this manuscript.

This project has received funding from the European High-Performance Computing Joint undertaking (Ju) under grant agreement No.~955701. The Ju receives
support from the European union’s Horizon 2020 research and innovation programme and Belgium, France, Germany, and Switzerland.
The authors declare no other conflicts of interest.

\bibliography{allshort}

\begin{thebibliography}{14}
\providecommand{\natexlab}[1]{#1}
\providecommand{\url}[1]{\texttt{#1}}
\expandafter\ifx\csname urlstyle\endcsname\relax
  \providecommand{\doi}[1]{doi: #1}\else
  \providecommand{\doi}{doi: \begingroup \urlstyle{rm}\Url}\fi

\bibitem[Arnold(1974)]{arnold_stochastic_1974}
L.~Arnold.
\newblock \emph{Stochastic differential equations: theory and applications}.
\newblock Wiley, New York, 1974.
\newblock ISBN 978-0-471-03359-2.

\bibitem[Blouza et~al.(2010)Blouza, Boudin, and Kaber]{blouza_parallel_2010}
A.~Blouza, L.~Boudin, and S.~M. Kaber.
\newblock Parallel in time algorithms with reduction methods for solving
  chemical kinetics.
\newblock \emph{Communications in Applied Mathematics and Computational
  Science}, 5\penalty0 (2):\penalty0 241--263, Dec. 2010.
\newblock ISSN 2157-5452, 1559-3940.
\newblock \doi{10.2140/camcos.2010.5.241}.

\bibitem[Bossuyt(2023)]{ignace_software_paper_ANZIAM}
I.~Bossuyt.
\newblock Code accompanying this paper: {Micro-macro Parareal, from ODEs to
  SDEs and back again}, 2023.
\newblock URL
  \url{https://gitlab.kuleuven.be/numa/public/micro-macro-Parareal-ANZIAM}.

\bibitem[Bossuyt et~al.(2023)Bossuyt, Vandewalle, and
  Samaey]{bossuyt_monte_carlomoments_2023}
I.~Bossuyt, S.~Vandewalle, and G.~Samaey.
\newblock Monte-{Carlo}/{Moments} micro-macro {Parareal} method for unimodal
  and bimodal scalar {McKean}-{Vlasov} {SDEs}, Oct. 2023.
\newblock URL \url{http://arxiv.org/abs/2310.11365}.
\newblock arXiv:2310.11365 [math.NA, physics, stat].

\bibitem[Gander and Vandewalle(2007)]{gander_analysis_2007}
M.~J. Gander and S.~Vandewalle.
\newblock Analysis of the parareal time-parallel time-integration method.
\newblock \emph{SIAM Journal on Scientific Computing}, 29\penalty0
  (2):\penalty0 556--578, 2007.
\newblock ISSN 10648275.
\newblock \doi{10.1137/05064607X}.

\bibitem[Gander et~al.(2023)Gander, Lunet, Ruprecht, and
  Speck]{gander_unified_2023}
M.~J. Gander, T.~Lunet, D.~Ruprecht, and R.~Speck.
\newblock A {Unified} {Analysis} {Framework} for {Iterative}
  {Parallel}-in-{Time} {Algorithms}.
\newblock \emph{SIAM Journal on Scientific Computing}, 45\penalty0
  (5):\penalty0 A2275--A2303, Oct. 2023.
\newblock ISSN 1064-8275, 1095-7197.
\newblock \doi{10.1137/22M1487163}.

\bibitem[Kloeden and Platen(1999)]{kloeden_platen_1999}
P.~E. Kloeden and E.~Platen.
\newblock \emph{Numerical solution of stochastic differential equations}.
\newblock Number~23 in Applications of mathematics. Springer, Berlin
  Heidelberg, 1999.
\newblock ISBN 978-3-540-54062-5 978-3-642-08107-1.

\bibitem[Legoll et~al.(2013)Legoll, Lelièvre, and Samaey]{Legoll2013}
F.~Legoll, T.~Lelièvre, and G.~Samaey.
\newblock {A micro-macro parareal algorithm: application to singularly
  perturbed differential equations}.
\newblock \emph{SIAM Journal on Scientific Computing, 2013-01}, 35\penalty0
  (4):\penalty0 p.A1951--A1986, 2013.
\newblock \doi{10.1137/120872681}.

\bibitem[Lions et~al.(2001)Lions, Maday, and
  Turinici]{lions_resolution_2001_AMS}
J.-L. Lions, Y.~Maday, and G.~Turinici.
\newblock R\'{e}solution d'{EDP} par un sch\'{e}ma en temps ``parar\'{e}el''.
\newblock \emph{C. R. Acad. Sci. Paris S\'{e}r. I Math.}, 332\penalty0
  (7):\penalty0 661--668, 2001.
\newblock ISSN 0764-4442.
\newblock \doi{10.1016/S0764-4442(00)01793-6}.
\newblock URL \url{https://doi.org/10.1016/S0764-4442(00)01793-6}.

\bibitem[Rackauckas and Nie(2017)]{rackauckas_differentialequationsjl_2017}
C.~Rackauckas and Q.~Nie.
\newblock {DifferentialEquations}.jl – {A} {Performant} and {Feature}-{Rich}
  {Ecosystem} for {Solving} {Differential} {Equations} in {Julia}.
\newblock \emph{Journal of Open Research Software}, 5\penalty0 (1):\penalty0
  15, May 2017.
\newblock ISSN 2049-9647.
\newblock \doi{10.5334/jors.151}.

\bibitem[Roberts(2014)]{roberts_model_2014}
A.~J. Roberts.
\newblock \emph{Model {Emergent} {Dynamics} in {Complex} {Systems}}.
\newblock Society for Industrial and Applied Mathematics, Philadelphia, PA,
  Jan. 2014.
\newblock ISBN 978-1-61197-355-6 978-1-61197-356-3.
\newblock \doi{10.1137/1.9781611973563}.

\bibitem[Rodriguez and Tuckwell(1996)]{rodriguez_statistical_1996}
R.~Rodriguez and H.~C. Tuckwell.
\newblock Statistical properties of stochastic nonlinear dynamical models of
  single spiking neurons and neural networks.
\newblock \emph{Physical Review E}, 54\penalty0 (5):\penalty0 5585--5590, Nov.
  1996.
\newblock ISSN 1063-651X, 1095-3787.
\newblock \doi{10.1103/PhysRevE.54.5585}.

\bibitem[Sznitman(1991)]{snitzman_1991}
A.-S. Sznitman.
\newblock Topics in propagation of chaos.
\newblock In P.-L. Hennequin, editor, \emph{Ecole d'{Eté} de {Probabilités}
  de {Saint}-{Flour} {XIX} — 1989}, volume 1464, pages 165--251. Springer
  Berlin Heidelberg, Berlin, Heidelberg, 1991.
\newblock ISBN 978-3-540-53841-7 978-3-540-46319-1.
\newblock \doi{10.1007/BFb0085169}.
\newblock Series Title: Lecture Notes in Mathematics.

\bibitem[Van~Kampen(1985)]{van_kampen_elimination_1985}
N.~Van~Kampen.
\newblock Elimination of fast variables.
\newblock \emph{Physics Reports}, 124\penalty0 (2):\penalty0 69--160, July
  1985.
\newblock ISSN 03701573.
\newblock \doi{10.1016/0370-1573(85)90002-X}.

\end{thebibliography}

\newpage
\appendix
\section{Proof of lemma \ref{solution_micro_macro_Parareal}}
\label{proof_micro_macro_error}
\begin{proof}\,
\begin{itemize}
\item First, we observe that, for the operators in definition \ref{definition_micro_macro_Parareal_ODEs}, it holds that 
for any $U$ and $u$, $\mathcal R (\mathcal M(U, u)) = U$ and
$\mathcal M ( \mathcal R (u), u) = u$.
Thus, as a result of property \ref{properties_micro_macro_Parareal}, we have 
$U^k_n = x^k_n$ for all $0 \leq n \leq N$.

\item In the zeroth iteration 
it holds that
\begin{equation}
\begin{aligned}
U^0_{n+1} &= G U^0_n \\
u^{0}_{n+1}
&= 
\mathcal L (U^0_{n+1})
=
\left[ \begin{matrix} 
G U^0_n \\ 0 \end{matrix} \right] 
=
\left[ \begin{matrix} 
G x^0_n \\ 0 \end{matrix} \right] 
= 
\mathcal B u^{0}_{n},
\end{aligned}
\end{equation}

\item In subsequent iterations, it holds that
\begin{equation}
\begin{aligned}
U^{k+1}_{n+1} 
&= 
\mathcal C (U^{k+1}_n) 
+ \mathcal{R} ( \mathcal F (x^k_n)) 
- \mathcal{C} (U^{k}_n) \\
&= 
G U^{k+1}_n 
+ F x^k_n + b y^k_n 
- G U^{k}_n 
\\
u^{k+1}_{n+1}
&= 
\mathcal M \left( 
U^{k+1}_{n+1}, \mathcal F(x^k_n) \right) 
= 
\mathcal M \left( 
U^{k+1}_{n+1}, \left[
\begin{matrix}
F x^k_n + by^n \\ d y^k_n
\end{matrix}
\right] \right) \\
&= 
\left[
\begin{matrix}
U^{k+1}_{n+1} \\ d y^k_n
\end{matrix}
\right]
= 
\left[ \begin{matrix} 
G U^{k+1}_n + F x^k_n + b y^k_n - G U^{k}_n \\ d y^k_n \end{matrix} \right] \\
&= 
\left[ \begin{matrix} 
G x^{k+1}_n + F x^k_n + b y^k_n - G x^{k}_n \\ d y^k_n \end{matrix} \right] \\
&= 
(\mathcal A - \mathcal B) u^k_n + \mathcal B u^{k+1}_n,
\end{aligned}
\label{Parareal_with_coupling_operators}
\end{equation}
where $d = e^{\delta \Delta t}$
\end{itemize}
\end{proof}

\section{Solution of a linear recursion}
\label{appendix_matrix_recursion}
\begin{lemma}[Solution of a matrix recursion]
\label{lemma_matrix_recursion}
Let $B, L \in \mathbb{R}^{d \times d}$ be matrices, let $\epsilon_0 \in \mathbb{R}^d$ be a vector and $b \in \mathbb R$ a scalar, then this iteration for $\mathbf e^{x,(k)} \in \mathbb{R}^d$,
\begin{equation}
\mathbf e^{x,(k)} = B \mathbf e^{(k-1)} + b M L^{k-1} \epsilon_0,
\label{statement_matrix_recursion}
\end{equation}
has the solution 
\begin{equation}
\begin{aligned}
\mathbf e^{x,(k)} 
&= 
B^{k} 
\mathbf e^{(0)} 
+
b 
\sum_{i=0}^{k-1} 
B^{i} 
M L^{k-1-i}
\epsilon_0.
\end{aligned}
\label{solution_matrix_recursion}
\end{equation}
\end{lemma}
\begin{proof}
The correctness of equation \eqref{solution_matrix_recursion} can be verified by filling it in equation \eqref{statement_matrix_recursion}. Alternatively, equation \eqref{solution_matrix_recursion} can be derived by solving the recursion that arises from equation \eqref{statement_matrix_recursion}.
\end{proof}


\section{Proof of lemma \ref{lemma_commuting_matrices}}
\label{appendix_proof_commuting_matrices}

\paragraph{Commutativity of $M^{-1}$ and $C$}
First we give an expression for $M^{-1}$ (the matrix $M$ is defined in equation \eqref{definition_matrices_M_and_C}). It as a strictly lower triangular Toeplitz matrix (see \mycite{gander_analysis_2007}):
\begin{equation}
M^{-1} = 
\left[
\begin{matrix}
I & 0 & 0 & \hdots & & 0\\
\mathcal{B} & I & 0 & \hdots & & 0\\
\mathcal{B}^2 & \mathcal{B} & I & \hdots & & 0 \\
\mathcal{B}^3 & \mathcal B^2 & \beta & \hdots & & 0\\
\vdots & & & & & \vdots \\
\mathcal{B}^{N-1} & \mathcal{B}^{N-2} & \mathcal{B}^{N-3} & \hdots & \mathcal{B} & I \\
\end{matrix}
\right].
\end{equation}
Let $\gamma = \mathcal A - \mathcal B$, then we have  that $M^{-1} C$ is a lower-triangular matrix
\begin{equation}
M^{-1} C = 
\left[
\begin{matrix}
0 & 0 & 0 & \hdots & & 0\\
\gamma & 0 & 0 & \hdots & & 0\\
\mathcal{B} \gamma & \mathcal{B} \gamma  & 0 & \hdots & & 0 \\
\mathcal{B}^2 \gamma & \mathcal{B} \gamma & \gamma & \hdots & & 0\\
\vdots & & & & & \vdots \\
\mathcal{B}^{N-2} \gamma & \mathcal{B}^{N-3} \gamma & \mathcal{B}^{N-4}\gamma & \hdots & \gamma & 0 \\
\end{matrix}
\right]
\end{equation}
On the other hand, we have that $C M^{-1}$ is a lower-triangular matrix
\begin{equation}
C M^{-1} = 
\left[
\begin{matrix}
0 & 0 & 0 & \hdots & & 0\\
\gamma & 0 & 0 & \hdots & & 0\\
\gamma \mathcal{B} & \gamma & 0 & \hdots & & 0 \\
\gamma \mathcal{B}^2 & \gamma \mathcal{B} & \gamma & \hdots & & 0\\
\vdots & & & & & \vdots \\
\gamma \mathcal{B}^{N-2} & \gamma \mathcal{B}^{N-3} & \gamma \mathcal{B}^{N-4} & \hdots & \gamma & 0 \\
\end{matrix}
\right]
\end{equation}
If $\gamma$ and $\beta$ do commute, then $M^{-1}$ and $C$ are commutative.

\paragraph{Proof of other statements} 
The fact that $M^{-1} C = HD$ can be readily checked.
Similarly, it can be verified that $C M^{-1} = DH$.
Thus, if $M^{-1} C = C M^{-1}$, then $HD = DH$. In words, if $M^{-1}$ and $C$ commute, so do $H$ and $D$.

\section{Proof of lemma \ref{lemma_matching_operator}}
\label{proof_property_matching}
\begin{proof}
We first prove that, under the conditions of lemma \ref{lemma_matching_operator}, it is true that $\mathcal{R}\left( \mathcal M( U, u) \right) = U$. 
Let $X$ be a random variable $X \in \mathbb R^d$ with $\mathrm{Cov}[X] = Q Q^T$ and $Q$ invertible, $M \in \mathbb R^d$ be a vector and $\Sigma \in \mathbb R^{d \times d}$ a matrix with $\Sigma = V V^T$, 
Then we need to prove that
\begin{equation}
\begin{aligned}
\mathbb{E}_P[ \mathcal{M}([M, \, \Sigma], X) ] &= M, \\
\mathrm{Cov}_P[ \mathcal{M}([M, \, \Sigma], X) ] &= \Sigma.
\end{aligned}
\label{property_matching_operator}
\end{equation}
 We name $Y = \mathcal{M}([M, \, \Sigma], X)$.
\begin{itemize}
\item Using the linearity of the expectation, we have
\begin{equation}
\mathbb{E}_P[Y] = \mathbb{E}_P[AX] - A \mathbb{E}_P [X] + M = M.
\label{proof_for_mean}
\end{equation}

\item Using that for a random variable $X \in \mathbb R^d$, with an arbitrary vector $b \in \mathbb R^d$ and matrix $A \in \mathbb R^{d \times d}$, $\mathrm{Cov}_P[A(X-b)] = A \mathrm{Cov}_P[X]A^T$, we obtain
\begin{equation}
\begin{aligned}
&\mathrm{Cov}_P[Y] 
= \mathrm{Cov}_P[A(X-\mathbb{E}_P[X])+M]
= A \mathrm{Cov}_P [X] A^T \\
&= A Q Q^T A^T = VV^T,
\end{aligned}
\label{proof_for_variance}
\end{equation} 
where we used the definition $A = V Q^{-1}$.
\end{itemize}
Equations \eqref{proof_for_mean} and \eqref{proof_for_variance} prove equation \eqref{property_matching_operator}.

\noindent
Now we prove that $\mathcal{M} \left( \mathcal R (u), u \right) = u$. In this case, $A = V Q^{-1} = I$ and $M = \mathbb{E}_P[\bar x])$.
\begin{equation}
\begin{aligned}
\mathcal{M}([\mathbb{E}_P[\bar x], \, \mathrm{Cov}_P[\bar x]], \bar x)  
= A (\bar x - \mathbb{E}_P[\bar x]) + M 
= \bar x.
\end{aligned}
\end{equation}
This ends the proof.
\end{proof}


\end{document}